\documentclass[11pt]{amsart}
\usepackage{geometry}                
\usepackage{graphicx}
\usepackage{amssymb}
\usepackage{epstopdf}
\usepackage{amsmath}
\usepackage{mathrsfs}
\newtheorem {theorem}    {Theorem}[section]

\newtheorem {lemma}      [theorem]    {Lemma}

\newtheorem {proposition}[theorem]    {Proposition}
\numberwithin{equation}{section}

\newtheorem*{thmA}{Theorem A}
\newtheorem*{thmB}{Theorem B}
\newtheorem*{thmC1}{Theorem C1}
\newtheorem*{thmC2}{Theorem C2}
\newtheorem*{thmD1}{Theorem D1}
\newtheorem*{thmD2}{Theorem D2}

\title{Archimedean Zeta Integrals on $U(2,1)$}

\author{Dongwen Liu}
\address{Department of Mathematics, University of Connecticut, Storrs, CT 06269, USA}
\email{dongwen.liu@uconn.edu}

\subjclass[2010]{Primary 22E45; Secondary 11F27, 11F67}

\thanks{This work was partially supported by NSFC11201384.}

\begin{document}
\maketitle

\begin{abstract}
It is known that for a dual pair of unitary groups with equal size, zeta integrals arising from Rallis inner product formula give the central values of certain automorphic $L$-functions, which have applications to arithmetic. In this paper we explicitly calculate archimedean zeta integrals of this type for $U(2,1)$. In particular we compute the matrix coefficients of Weil representations using joint harmonics in the Fock model, and those of discrete series using Schmid operators.
\end{abstract}

\section{Introduction}

In the theory of automorphic forms, to obtain deep arithmetic applications it is often necessary to have effectively computable results locally at each place of a number field. In particular, at archimedean places we need some detailed calculations on the harmonic analysis of real Lie groups, and have to face certain combinatorial complexity.

In \cite{L1, L2} the author works out the local archimedean theory for cohomological representations and the connection to special values of $L$-functions in the stable range. The local theory of theta correspondence for unitary groups is developed in \cite{HKS} for $p$-adic fields and in \cite{P} for field of real numbers. It is made more precise in 
\cite{H1, H2} and applied to obtain results about special values of $L$-functions. The article \cite{HLS} outlines the general global theory of theta correspondence for cuspidal automorphic representations of unitary groups, in which the authors obtain a general formula for the local zeta integral on unitary groups and show that the central $L$-value is non-negative.

We shall explain the motivation and background of this paper by a brief review for above context, following \cite{HLS}. Let $F^+$ be a totally real number field, $F$ a totally imaginary quadratic extension of $F^+$, ${\bf A}={\bf A}_{F^+}$ the adele ring of $F^+$. Let $V$ (resp. $V'$) be a hermitian (resp. skew-hermitian) vector space of dimension $n$ over $F$, and $W=V\otimes_{F}V'$, a
symplectic space over $F^+$. Fixing an additive character $\psi$ and a complete polarization $W=X\oplus Y$, we have the Schr\"{o}dinger model of the
oscillator representation $\omega_\psi$ of $\widetilde{Sp}(W)({\bf A})$, realized on the space ${\mathcal S}(X({\bf A}))$ of Schwartz-Bruhat functions on $X({\bf A})$.
Let $G=U(V),$ $G'=U(V')$. By choosing a global splitting character $\chi$ of ${\bf A}_{F}^\times/F^\times$ (cf. \cite{HLS}), $\omega_\psi$ then defines a Weil representation $\omega_{V,V',\chi}$ of $G({\bf A})
\times G'({\bf A})$ on ${\mathcal S}(X({\bf A}))$. As usual, for $\Phi\in {\mathcal S}(X({\bf A}))$ we have the theta kernel
\[
\theta_\Phi(g,g')=\sum_{\gamma\in X(F^+)}\left(\omega_{V,V',\chi}(g,g')\Phi\right)(\gamma).
\]
For a cusp form $f$ on $G(F^+)\backslash G({\bf A})$ define
\[
\theta_\Phi(f)(g')=\int_{G(F^+)\backslash G({\bf A})}\theta_\Phi(g,g')f(g)dg,\quad g'\in G'({\bf A}).
\]

We now recall the Rallis inner product formula and the construction of the standard $L$-functions
for unitary groups. Let $\pi$, $\tilde{\pi}$ be cuspidal automorphic representations of $G$, $f\in\pi$, $f\in\tilde{\pi}$.
Let $H=U(V\oplus (-V))$ be the quasi-split unitary group of size $2n$, $i_V: G\times G\hookrightarrow H$ 
be the natural inclusion, following the doubling method. The Piatetski-Shapiro-Rallis zeta integral is then defined by
\begin{equation}
Z(s, f, \tilde{f},\varphi,\chi)=\int_{(G\times G)(F^+)\backslash (G\times G)({\bf A})}E(i_V(g,\tilde{g}), s, \varphi, \chi)f(g)\tilde{f}(\tilde{g})\chi^{-1}\circ\det(\tilde{g})dg d\tilde{g},
\end{equation}
where $E(\cdot, s,\varphi, \chi)$ is the Eisenstein series on $H({\bf A})$ as in \cite[\S 1]{H2}, and $\varphi=\varphi(s)$ is a section of a degenerate principal series
$I_n(s,\chi)$ varying in $s$. This integral converges absolutely for $Re(s)$ large enough and admits an Euler expansion if $\varphi$, $f$ and $\tilde{f}$ are factorizable. 
It vanishes unless $\tilde{\pi}\cong\pi^\vee$, which we assume in the sequel.

For cusp forms $f$, $\tilde{f}$ on $G({\bf A})$ define the paring
\[
\langle f, \tilde{f}\rangle =\int_{G(F^+)\backslash G({\bf A})}f(g)\tilde{f}(g)dg.
\]
Write $\pi =\otimes \pi_v$, $\pi^\vee=\otimes\pi_v^\vee$ and choose compatible invariant local parings $\langle\cdot,\cdot\rangle_v$ between $\pi_v$ and $\pi_v^\vee$ such that
\[
\langle f, \tilde{f}\rangle=\prod_v \langle f_v,\tilde{f}_v\rangle_v.
\]
There is a conjugate linear isomorphism $\pi_v\cong \pi_v^\vee$, $f_v\mapsto \bar{f}_v$ at each place $v$.

Choose $\phi=\otimes\phi_v\in {\mathcal S}(X({\bf A}))$. For $S$ a sufficiently large finite set of places of $F^+$, including archimedean ones, we have the Euler product
\[
Z(s, f, \tilde{f}, \varphi, \chi)=\prod_{v\in S}Z(s, f_v, \tilde{f}_v, \varphi_v, \chi_v)d_n^S(s)^{-1}L^S(s+\frac{1}{2},\pi, St,\chi),
\]
where $L^S(s+\frac{1}{2},\pi, St,\chi)$ is the partial $L$-function of $\pi$ twisted by $\chi$, attached to $2n$-dimensional standard representation of the $L$-group, and
$d_n^S(s)$ is a product of certain partial $L$-functions attached to the extension $F/F^+$ (cf. \cite{H2}). Assume now $\tilde{\pi}\cong\pi^\vee$, then after proper 
normalization the {\it Rallis inner product
formula} can be written as
\begin{equation}
\langle\theta_\phi(f),\theta_{\bar{\phi}}(\bar{f})\rangle=\prod_{v\in S}Z(0, f_v,\bar{f}_v,\varphi_v,\chi_v)d_n^S(0)^{-1}L^S(\frac{1}{2},\pi, St,\chi),
\end{equation}
where $\varphi=\delta(\phi\otimes\bar{\phi})$ in the notation of \cite[p.182]{L2}.

The central $L$-value $L(\frac{1}{2},\pi, St,\chi)$ is of great interest and has a bunch of applications to arithmetic (cf. \cite{HLS}). In particular, knowing some information about the explicit local value at each place is quite useful. In \cite{HLS} under certain assumptions it is shown that
$L^S(\frac{1}{2},\pi, St,\chi)\geq 0$ for any finite set $S$ of places of $F^+$. More precisely, non-negativity of local central values is established for the cases (i) $v$ real, $\pi_v$ in the discrete series (our standing hypothesis); (ii) $v$ finite, $F/F^+$ split at $v$ or $\pi_v$ tempered.

As explained in \cite[\S2]{L2}, one has
\[
Z(0,f_v, \bar{f}_v, \varphi_v,\chi_v)=\int_{G(F_v^+)}(\omega_{\chi_v} (g)\phi_v, \phi_v) (\pi_v(g)f_v, f_v)dg,
\]
which integrates matrix coefficient of the Weil representation against that of $\pi_v$. From now on we assume that $v$ is real, $\pi_v$ is in the discrete series, $\phi_v$ is in the space of joint harmonics, and we shall replace
$(\pi_v(g)f_v, f_v)$ by a canonical matrix coefficient $\psi_{\pi_v}(g)$ of $\pi_v$ (see Section 4 for details). Our purpose in this paper  is to explicitly compute the archimedean zeta integral
\begin{equation} \label{zeta'}
\int_{G(F_v^+)}(\omega_{\chi_v}(g)\phi_v,\phi_v)\cdot \psi_{\pi_v}(g)dg
\end{equation}
in the case that $G(F_v^+)$ is the real unitary group $U(2,1)$. The case $G(F_v^+)=U(1,1)$ is solved in the dissertation \cite{Lin}, and we are interested in the first noncompact example of larger size, namely, the group $U(2,1)$. 

The main results of this paper can be formulated as follows. Fix an additive character $\psi$ of $\mathbb{R}$. Let $V$ be a 3-dimensional complex
Hermitian space, and let $G$ denote the unitary group attached to $V$. For each
odd-dimensional complex skew-Hermitian space $V'$, the group $G$ is a subgroup of
the real symplectic group Sp$(V\otimes_\mathbb{C} V')$ as usual. Define the metaplectic double cover $\widetilde{G}$
of $G$ to be the double cover of $G$ induced by the metaplectic double
cover $\widetilde{\textrm{Sp}}(V\otimes_\mathbb{C}V')\to \textrm{Sp}(V 
\otimes_\mathbb{C}V')$. This is independent of $V'$. Let $\lambda_1$, $\lambda_2$, 
$\lambda_3$ be
three pairwise different elements of $\frac{1}{2}+\mathbb{Z}$.
Let $\pi_\lambda$ be the genuine discrete series
representation of $\widetilde{G}$
with Harish-Chandra parameter $\lambda:=(\lambda_1,\lambda_2, \lambda_3)$, to be viewed as an irreducible unitary representation. By theta dichotomy for real unitary groups \cite{P} and a result in \cite{L1} on discrete spectrum of local theta correspondence, up to isometry there exists a unique 3-dimensional skew-Hermtian space $V'$ such that $\pi_\lambda$ occurs as a subrepresentation in the unitary oscillator representation $\omega_{V,V',\psi}$ attached to $V$, $V'$ and
$\psi$. Let $P_\lambda:\omega_{V,V',\psi}\to \omega_{V,V',\psi}$ denote the orthogonal projection to the $\pi_\lambda$-isotypic subspace.
Fix a maximal compact subgroup $K$ of $G$, which induces a maximal compact
subgroup $\widetilde{K}$ of $\widetilde{G}$.
Denote by $\tau_\lambda$ the minimal $\widetilde{K}$-type of $\pi_\lambda$. Then there is a positive number $c_{\psi, V, \lambda}$ such that
\[
\| P_\lambda(\phi)\| = c_{\psi, V, \lambda} \|\phi\|
\]
for all $\phi$ which lays in the  $\tau_\lambda$-isotypic subspace of the space of joint harmonics
(with respect to $K$ and an arbitrary maximal compact subgroup of the unitary
group attached to $V'$). The constant $c_{\psi,V,\lambda}$ is 1 when either $V$ or $V'$ is anisotropic. The main result of this paper is equivalent to an explicit
calculation of $c_{\psi, V, \lambda}$ when $V$ is of signature $(2,1)$ and $\psi$ is chosen to be $\psi_a: t\mapsto e^{2\pi i a t}$ for some $a>0$. In this case we list the explicit values of $c_{\psi,V,\lambda}$ below (cf. Section 5, Theorem A to D2).

\begin{theorem} Follow above notations, assume that $V$ has signature $(2,1)$, $\psi=\psi_a$ for some $a>0$ and let $\mu$'s, $\nu$'s, $\alpha$'s and $\beta$'s  below stand for non-negative integers. Then

(1) if $\lambda=(-\mu_2-\frac{1}{2}, -\mu_1-\frac{3}{2}, \nu+\frac{1}{2})$, $\mu_1\geq\mu_2$ or $\lambda=(\nu_1+\frac{3}{2}, \nu_2+\frac{1}{2}, -\alpha-\frac{1}{2})$, $\nu_1\geq\nu_2$, then $V'$ is anisotropic and $c_{\psi, V, \lambda}=1$.

(2) if $\lambda=(-\mu_2-\frac{1}{2},-\mu_1-\frac{3}{2}, -\alpha+\frac{3}{2})$, $\alpha-4\geq \mu_1\geq\mu_2$, then
\[
c^2_{\psi, V,\lambda}=\frac{(\alpha-\mu_1-3)(\alpha-\mu_2-2)}{(\alpha-1)\alpha}.
\]

(3) if $\lambda=(-\mu_2+\frac{1}{2}, -\mu_1-\frac{1}{2}, -\alpha-\frac{1}{2})$, $\mu_1\geq \mu_2\geq \alpha+2$, then
\[
c^2_{\psi, V,\lambda}=\frac{(\mu_2-\alpha-1)(\mu_1-\alpha)}{(\mu_2+\alpha)(\mu_1+\alpha+1)}.
\]

(4) if $\lambda=(-\mu_2-\frac{1}{2}, -\mu_1-\frac{1}{2}, -\alpha+\frac{1}{2})$, $\mu_1\geq \alpha\geq \mu_2+2$, then
\[
c^2_{\psi, V,\lambda}=\frac{(\mu_1-\mu_2)(\mu_1-\alpha+1)(\alpha-\mu_2-1)}{(\mu_1-\mu_2+1)(\mu_1+1)\alpha}.
\]

(5) if $\lambda=(\nu+\frac{1}{2}, -\mu-\frac{1}{2}, \beta-\frac{1}{2})$, $\beta\geq\nu+2$, then
\[
c^2_{\psi, V,\lambda}=\frac{(\beta+\mu)(\beta-\nu-1)}{(\beta+\mu+1)\beta}.
\]

(6) if $\lambda=(\nu-\frac{1}{2}, -\mu-\frac{1}{2}, \beta+\frac{1}{2})$, $\nu\geq\beta+2$, then
\[
c^2_{\psi, V,\lambda}=\frac{(\nu-\beta-1)(\nu+\mu)}{\nu(\nu+\mu+1)}.
\]

(7)  if $\lambda=(\nu_1+\frac{1}{2},\nu_2-\frac{1}{2},\beta+\frac{1}{2})$, $\nu_1\geq\nu_2\geq \beta+2$, then
\[
c^2_{\psi, V,\lambda}=\frac{(\nu_1-\beta)(\nu_2-\beta-1)}{(\nu_1+1)\nu_2}.
\]

(8) if $\lambda=(\nu_1+\frac{3}{2},\nu_2+\frac{1}{2},\beta-\frac{3}{2})$, $\beta-4\geq\nu_1\geq\nu_2$, then
\[
c^2_{\psi, V,\lambda}=\frac{(\beta-\nu_1-3)(\beta-\nu_2-2)}{(\beta-1)\beta}.
\]

(9) if $\lambda=(\nu_1+\frac{1}{2},\nu_2+\frac{1}{2}, \beta-\frac{1}{2})$, $\nu_1\geq\beta\geq \nu_2+2$, then
\[
c^2_{\psi, V,\lambda}=\frac{(\nu_1-\nu_2)(\nu_1-\beta+1)(\beta-\nu_2-1)}{(\nu_1-\nu_2+1)(\nu_1+1)\beta}.
\]

(10) if $\lambda=(\nu+\frac{1}{2},-\mu-\frac{1}{2}, -\alpha+\frac{1}{2})$, $\alpha\geq\mu+2$, then
\[
c^2_{\psi, V,\lambda}=\frac{(\alpha+\nu)(\alpha-\mu-1)}{(\alpha+\nu+1)\alpha}.
\]

(11) if $\lambda=(\nu+\frac{1}{2},-\mu+\frac{1}{2}, -\alpha-\frac{1}{2})$, $\mu\geq\alpha+2$, 
then
\[
c^2_{\psi, V,\lambda}=\frac{(\mu-\alpha-1)(\mu+\nu)}{\mu(\mu+\nu+1)}.
\]
\end{theorem}

We now describe the organization of the paper as well as the main tools and techniques used in our computation.  In Section 2 we review the local theta correspondence of real unitary groups, the Fock model for the Weil representation, and the important notion of joint harmonics introduced in \cite{Ho}. The explicit description for the theta correspondence between a dual pair of real unitary groups with equal size is given in \cite{L1, P}. After these preparations, starting from understanding the Lie group structures, in Section 3 we pick up the joint harmonics explicitly and then fully compute the matrix coefficients of Weil representations restricted on $U(2,1)$.

Section 4 treats the matrix coefficients of discrete series representations of $U(2,1)$, which are of independent interest. The main tool we use is the so-called Schmid operator 
\cite{S}, which has been explored to study the matrix coefficients of discrete series of $SU(2,2)$ (cf. \cite{HaKO, Y}), as well as Whittaker functions \cite{T} and 
Shintani functions \cite{Tsu} for rank 1 real unitary groups. The key idea is that the matrix coefficients of discrete series are annihilated by certain Schmid operators. Consequently, radial part of these functions are found to satisfy certain Riemann's differential equations, whose solutions involve hypergeometric functions, an extremely important family of special functions.

Combining the work in Section 3 and Section 4, eventually we are able to evaluate the zeta integral (\ref{zeta'}).  Overall, we have a bunch of possibilities for $G'$ and we also need to take care of different types of discrete series, namely, the (anti)-holomorphic discrete series and middle discrete series. It turns out that a great deal of combinatorial complexity pops out in our computation. The main results are given through Theorem A to D2 in Section 5. In all cases, the final expression for the explicit value  looks very neat, as can be seen from Theorem 1.1, although we apply a lot of machineries and carry out heavy calculations. 

The results of this paper might be useful for the study of certain two-dimensional Shimura varieties (cf. \cite{HLS}). In future work, we hope to address the applications to arithmetic.

{\bf Acknowledgement.} The author would like to thank J.-S. Li for suggesting this problem, thank B. Lin for sending him the dissertation \cite{Lin}, and thank A. W. Knapp for useful suggestions on computing matrix coefficients of discrete series. He also thanks the referee for numerous comments which help revise the paper.

\section{Preliminaries}

\subsection{Dual pairs  and joint harmonics} 

For $p, q\geq 0$, set
\[
I_{p,q}=\begin{bmatrix} I_p & 0\\ 0 & -I_q\end{bmatrix},
\]
where $I_p$ is the $p\times p$ identity matrix. If $g\in M_{p+q}(\mathbb{C})$, set $g^*$ equal to the conjugate transpose matrix of $g$.
Let $U(p,q)$ be the group of all $g\in M_{p+q}(\mathbb{C})$ such that
\[
g I_{p,q}\cdot g^*= I_{p,q}.
\]

Let $V$, $V'$ be complex vector spaces with hermitian and skew-hermitian forms $I_{p,q}$ and $\sqrt{-1}I_{r,s}$. Let $G, G'$ be the isometry groups of $V$ and $V'$ respectively. Then $(G,G')=(U(p,q),U(r,s))$ is an irreducible dual pair inside $Sp(W)=Sp_{2N}(\mathbb{R})$, where $W=V\otimes V'$, and $N=(p+q)(r+s)$.

Let $\{\xi_1,\ldots, \xi_p, \eta_1,\ldots \eta_q\}$ and $\{e_1,\ldots e_r, f_1, \ldots f_s\}$ be the standard basis of $V$ and $V'$ respectively.  Then
\[
\{\zeta_1, \ldots, \zeta_N\}:=\{\xi_1\otimes e_1,\ldots, \xi_1\otimes f_s, \ldots, \eta_q\otimes e_1, \ldots,
\eta_q\otimes f_s\}
\]
form a  basis of $W$ over $\mathbb{C}$.
Let $J=I_{p,q}\otimes I_{r,s}.$
Under the basis
\[
(\zeta_1,\ldots, \zeta_N, \sqrt{-1}\zeta_1,\ldots, \sqrt{-1}\zeta_N)\begin{bmatrix}
I_N & 0\\
0 & J
\end{bmatrix}
\]
of $W$ over $\mathbb{R}$, the symplectic form on $W$ is represented by the matrix
\[
\begin{bmatrix}
0 & I_N\\
-I_N & 0
\end{bmatrix}.
\]
With the basis of $V$, $V'$ and $W$ specified as above, we have the following embeddings
\begin{equation}
\label{gembed} 
\left\{
\begin{aligned}
&\iota: G\hookrightarrow Sp(W), \qquad X+\sqrt{-1}Y\mapsto\begin{bmatrix}
X\otimes I & (Y\otimes I)J\\
-J(Y\otimes I) & J(X\otimes I)J
\end{bmatrix}, \\
& \iota': G'\hookrightarrow Sp(W),\qquad X'+\sqrt{-1}Y'\mapsto\begin{bmatrix}
I\otimes X' & (I\otimes Y')J\\
-J(I \otimes Y') & J(I \otimes X')J
\end{bmatrix}.
\end{aligned}\right.
\end{equation}

We shall recall the notion of joint harmonics from \cite{Ho}. Let  $K=U(p)\times U(q)$, $K'=U(r)\times U(s)$ and $U=U(N)$ be the maximal compact subgroups of $G$, $G'$ and $Sp(W)$ with Lie algebras $\frak{k}$, $\frak{k}'$ and $\frak{u}$ respectively. Let $\frak{sp}$ be the Lie algebra of $Sp(W)$ and $\frak{sp}_{\mathbb{C}}=\frak{sp}(2N,\mathbb{C})$ be the complexification of $\frak{sp}$. Then we have
the Cartan decomposition
\[
\frak{sp}_{\mathbb{C}}=\frak{u}_{\mathbb{C}}\oplus \frak{p}_+ \oplus \frak{p}_-,
\]
where
\begin{align*}
&\frak{u}=\bigg\{\begin{bmatrix} A & B\\ -B & A\end{bmatrix}:{}^tA=-A, {}^tB=B, A, B\in M_N(\mathbb{R})\bigg\},\\
&\frak{u}_{\mathbb{C}}\cong\frak{gl}(n,\mathbb{C}),\\
&\frak{p}_+=\bigg\{\begin{bmatrix} A & \sqrt{-1}A\\ \sqrt{-1}A & -A\end{bmatrix}: {}^tA=A, A\in M_N(\mathbb{C})\bigg\},\\
&\frak{p}_-=\bigg\{\begin{bmatrix} A & -\sqrt{-1}A\\ -\sqrt{-1}A & -A\end{bmatrix}: {}^tA=A, A\in M_N(\mathbb{C})\bigg\}.
\end{align*}
Denote the centralizer of $\frak{k}$(resp. $\frak{k}'$) in $\frak{sp}$ by $\frak{m}'$(resp. $\frak{m}$), then we have the decompositions
\begin{align*}
&\frak{m}'_{\mathbb{C}}=\frak{m}'^{(1,1)}\oplus\frak{m}'^{(2,0)}\oplus \frak{m}'^{(0,2)},\\
&\frak{m}_{\mathbb{C}}=\frak{m}^{(1,1)}\oplus\frak{m}^{(2,0)}\oplus \frak{m}^{(0,2)},
\end{align*}
where $\frak{m}'^{(1,1)}=\frak{m}'_{\mathbb{C}}\cap \frak{u}_{\mathbb{C}}$,
$\frak{m}'^{(2,0)}=\frak{m}'_{\mathbb{C}}\cap \frak{p}_+$,
$\frak{m}'^{(0,2)}=\frak{m}'_{\mathbb{C}}\cap \frak{p}_-$, and similarly
$\frak{m}^{(1,1)}=\frak{m}_{\mathbb{C}}\cap \frak{u}_{\mathbb{C}}$,
$\frak{m}^{(2,0)}=\frak{m}_{\mathbb{C}}\cap \frak{p}_+$,
$\frak{m}^{(0,2)}=\frak{m}_{\mathbb{C}}\cap \frak{p}_-$. Fixing a non-trivial additive character $\psi$ of $\mathbb{R}$, the Harish-Chandra module $\omega^{HC}_\psi$ of the metaplectic representation $\omega_{\psi}$ of $\widetilde{Sp}_{2N}(\mathbb{R})$
can be realized on $\mathscr{P}_N$, the space of polynomials on $\mathbb{C}^N$. We define the $K$-harmonic space
\[
\mathscr{H}(K)=\{P\in\mathscr{P}_N: \ell(P)=0\textrm{ for all }\ell\in \frak{m}'^{(0,2)}\},
\]
and similarly the $K'$-harmonic space $\mathscr{H}(K')$ which consists of polynomials annihilated by $\frak{m}^{(0,2)}$. Then the space of joint harmonics is defined to be $\mathscr{H}=\mathscr{H}(K)\cap\mathscr{H}(K')$, which is $(\widetilde{K}\times\widetilde{K}')$-invariant, where $\widetilde{K}$ (resp. $\widetilde{K}'$) is the double cover of $K$ (resp. $K'$) in $\widetilde{Sp}_{2N}(\mathbb{R})$.

If $\sigma$ is a $\widetilde{K}$-type or $\widetilde{K}'$-type occurring in $\mathscr{P}_N$, then the {\it degree of $\sigma$} is the least degree of
nonzero polynomials occurring in the $\sigma$-isotypic component of $\mathscr{P}_N$. For the dual pair $(G,G')$ we know that
$\widetilde{K}$ and $\widetilde{K}'$ generate mutual commutants on $\mathscr{H}$,
\[
\mathscr{H}=\bigoplus_i\mathscr{H}_{\sigma_i, \sigma_i'}\cong \bigoplus_i \sigma_i\otimes\sigma_i',
\]
where for each $i$, $\sigma_i\in\widehat{\widetilde{K}}$ and $\sigma_i'\in\widehat{\widetilde{K}'}$ determine each other, and they have the same degree. 

From now on assume that $p+q=r+s$. Fix the character $\psi:  t\mapsto e^{2\pi i t}$ of $\mathbb{R}$.  One may fix a splitting character $\chi$ of $\mathbb{C}^\times/\mathbb{R}_+^\times$ as well such that $\chi|_{\mathbb{R}^\times}=\textrm{sgn}(\cdot)^{p+q}$, where $\textrm{sgn}(\cdot)$ is the sign character of $\mathbb{R}^\times$. For example one may take $\chi(z)=(z/|z|)^{p+q}$.

One has the following results from \cite{L1, P}.
Let $\pi$ and $\pi'$ be genuine irreducible admissible representations of $\widetilde{G}$ and $\widetilde{G}'$ respectively. If $\pi=\theta(\pi')$ under the theta correspondence, then every minimal $\widetilde{K}$-type of $\pi$ is of minimal degree and corresponds to a minimal $\widetilde{K}'$-type of $\pi'$ in the space of joint harmonics. Below, the highest weight of a representation of $K$ will be used to denote the corresponding representation of $\widetilde{K}$ by nontrivial extension across $\{\pm1\}$. 
The corresponding pairs $\sigma\in \widehat{\widetilde{K}}$, $\sigma'\in\widehat{\widetilde{K}'}$ are given by
\[
\left\{
\begin{aligned}
&\sigma=[(\mu_1,\ldots,\mu_m, 0,\ldots,0,-\nu_n,\ldots,-\nu_1)+\det{}^{(r-s)/2}]\\
&\qquad \otimes[(\alpha_1,\ldots,\alpha_k,0,\ldots,0,-\beta_l,\ldots,-\beta_1)+\det{}^{-(r-s)/2}],\\
&\sigma'=[(\mu_1,\ldots,\mu_m, 0,\ldots,0,-\beta_l,\ldots,-\beta_1)+\det{}^{(p-q)/2}]\\
&\qquad\otimes[(\alpha_1,\ldots,\alpha_k,0,\ldots,0,-\nu_n,\ldots,-\nu_1)+\det{}^{-(p-q)/2}]
\end{aligned}\right.
\]
where $\mu_1\geq\cdots\geq\mu_m>0,$ $\nu_1\geq\cdots\geq\nu_n>0$, $\alpha_1\geq\cdots\geq\alpha_k>0$, $\beta_1\geq\cdots\geq\beta_l>0$. The indices satisfy the obvious constraints $m+n\leq p$, $k+l\leq q$, $m+l\leq r$, $k+n\leq s$. In general, if a $\widetilde{K}$-type
\[
\tau=[(x_1,\ldots, x_p)+ \det{}^{(r-s)/2}]\otimes [(y_1,\ldots, y_q)+\det{}^{-(r-s)/2}]
\]
occurs in ${\mathscr H}$, then by \cite[Lemma 1.4.5]{P} the degree of $\tau$ is
\[
\sum^p_{i=1}|x_i| +\sum^q_{j=1} |y_j|.
\]
In particular this verifies that $\sigma$ and $\sigma'$ above have the same degree. With $\psi$ being fixed as above, we will then drop it from the notation and write $\omega_\psi$, $\omega^\infty_\psi$, $\omega^{HC}_\psi$ etc. as $\omega$, $\omega^\infty$,  $\omega^{HC}$ respectively.

\subsection{The Fock model}

We follow the treatment in \cite{F}. The smooth representation $\omega^\infty$ corresponding to $\omega$ can be realized as the Fock space
\[
{\mathscr F}_N=\{f: f \textrm{ is entire on }\mathbb{C}^N  \textrm{ and } \|f\|_{\mathscr{F}_N}^2=\int_{\mathbb{C}^N} |f(z)|^2 e^{-\pi|z|^2}dz<\infty\}.
\]
Let $z={}^t(z_1,\ldots, z_N)$ be the variables of polynomials in $\mathscr{P}_N$. Then $\left\{\sqrt{\frac{\pi^{|\alpha|}}{\alpha!}}z^\alpha, |\alpha|\geq 0\right\}$ forms an orthonormal basis for ${\mathscr F}_N$.
The Bargmann transform gives an isometry from $L^2(\mathbb{R}^N)$, the space on which the Schr\"{o}dinger representation is realized, onto ${\mathscr F}_N$. Based
on this transform, we introduce the following linear map
\begin{equation}\label{gc}
M_{2N}(\mathbb{R})\to M_{2N}(\mathbb{C}),\quad g=\begin{bmatrix} A & B\\ C & D\end{bmatrix}\mapsto g^c=\begin{bmatrix} \frac{A+D}{2}+\sqrt{-1}\frac{C-B}{2} &
\frac{A-D}{2}+\sqrt{-1}\frac{C+B}{2} \\ \frac{A-D}{2}-\sqrt{-1}\frac{C+B}{2} & \frac{A+D}{2}-\sqrt{-1}\frac{C-B}{2}\end{bmatrix}.
\end{equation}
Denote by $Sp^c$ the image of $Sp_{2N}(\mathbb{R})$. Let $\nu$ be the Fock projective representation of $Sp^c$ on ${\mathscr F}_N$. Then for
$g^c= \begin{bmatrix} P & Q\\ \overline{Q} & \overline{P}\end{bmatrix} \in Sp^c$,  up to a factor of $\pm 1$ the operator $\nu(g^c)$ is given by
\begin{equation}\label{fockaction}
\left\{
\begin{aligned}
&\nu(g^c)f(z) = \int_{\mathbb{C}^N}K_{g^c}(z, \bar{w})f(w)e^{-\pi |w|^2} dw,\\
&K_{g^c}(z,\bar{w})= (\det P)^{-\frac{1}{2}} \exp\bigg[\frac{\pi}{2}\big({}^tz\overline{Q}P^{-1}z+2{}^t\bar{w} P^{-1}z -{}^t\bar{w}P^{-1} Q \bar{w}\big)\bigg].
\end{aligned}\right.
\end{equation}
The unitary group $U(N)$ embeds into $Sp^c$ by
\[
U(N)\hookrightarrow Sp^c,\quad P\mapsto\begin{bmatrix} P & 0 \\ 0 & \overline{P}\end{bmatrix},
\]
with the action of the Fock representation
\begin{equation}\label{nuu}
\nu\begin{bmatrix} P & 0 \\ 0 & \overline{P}\end{bmatrix}f(z)=(\det P)^{-\frac{1}{2}}f(P^{-1}z).
\end{equation}
For convenience we shall also label $z_1, \ldots, z_N$ as $z_{11},\ldots, z_{1n}, \ldots, z_{n1},\ldots, z_{nn}$, where $n=p+q=r+s$.
If we write 
\[
z=\begin{bmatrix} z_{11} & \cdots & z_{1n} \\ & \cdots & \\ z_{n1} & \cdots & z_{nn}\end{bmatrix}=\begin{bmatrix} A_{p\times r} & B_{p\times s} \\ 
C_{q\times r} & D_{q\times s}\end{bmatrix}, 
\]
then from (\ref{gembed}), (\ref{gc}) and (\ref{nuu}) it follows that for $k=(k_1,k_2)\in K=U(p)\times U(q)$,
\begin{equation}\label{fockk}
\omega(k)f(z)=(\det k_1)^{\frac{r-s}{2}}(\det k_2)^{-\frac{r-s}{2}}f\begin{pmatrix}
{}^tk_1 A & k_1^{-1}B \\ k_2^{-1}C & {}^tk_2 D\end{pmatrix}.
\end{equation}

Let $d\nu$ be the infinitesimal Fock representation of $\frak{sp}$ on $\mathscr{P}_N$. We collect some formulas for the action of $d\nu$ from \cite{F}, which will be used to
determine the joint harmonics explicitly in the next section.
 Denote by $p_-(A)$ the following typical element in $\frak{p}_-$,
\[
p_-(A)=\begin{bmatrix} A & -\sqrt{-1}A\\ -\sqrt{-1}A & -A\end{bmatrix},
\]
where $^tA=A$, $A\in M_N(\mathbb{C})$. Then restricted on $\frak{p}_-$ one has
\begin{equation}\label{p-}
 d\nu[p_-(X_{ij}+X_{ji})]=\frac{2}{\pi}\frac{\partial^2}{\partial z_i\partial z_j},\quad i,j=1,\ldots, N,
\end{equation}
where the $X_{ij}$'s are elementary matrices given by
\[
X_{ij}=(\delta_{ik}\delta_{jl})_{1\leq k, l\leq N}\quad \textrm{with Kronecker's delta }\delta_{ik}.
\]
We also have another formula
\begin{equation}\label{cartan}
d\nu(X_{i,N+i}-X_{N+i,i})=\sqrt{-1}\bigg(z_i\frac{\partial}{\partial z_i}+\frac{1}{2}\bigg), \quad i=1,\ldots, N.
\end{equation}
Let $\frak{t}\subset \frak{k}$ and $\frak{t}'\subset\frak{k}'$ be the compact Cartan subalgebras which consist of $n\times n$ diagonal matrices with purely imaginary diagonal entries. Let $(\varepsilon_1, \ldots, \varepsilon_{n})$ and $(\varepsilon_1',\ldots,\varepsilon_{n}')$ denote the signatures of $I_{p,q}$ and $I_{r,s}$ respectively, i.e. $(\underbrace{+,\ldots,+}_p,\underbrace{-,\ldots,-}_q)$ and $(\underbrace{+,\ldots,+}_r,\underbrace{-,\ldots,-}_s)$. Then from either (\ref{fockk}) or (\ref{cartan}) we can see that
\begin{equation}
\label{taction}
\left\{
\begin{aligned}
&d\nu \big[\iota \cdot {\rm diag}(t_1, \ldots, t_n)\big]=\sum^n_{i=1}\varepsilon_i t_i\bigg(\sum^n_{j=1}\varepsilon'_jz_{ij}\frac{\partial}{\partial z_{ij}}+\frac{r-s}{2}\bigg),\\
&d\nu\big[\iota' \cdot {\rm diag}(t'_1, \ldots, t'_n)\big]=\sum^n_{j=1}  \varepsilon_j't'_j\bigg(\sum^n_{i=1}\varepsilon_iz_{ij}\frac{\partial}{\partial z_{ij}}+\frac{p-q}{2}\bigg),
\end{aligned}\right.
\end{equation}
where ${\rm diag}(t_1, \ldots, t_n)\in\frak{t}$,
${\rm diag}(t'_1, \ldots, t'_n)\in\frak{t}'$.

\section{Matrix coefficients of the Weil representations}

\subsection{Structures of $U(2,1)$}

Let $G=U(2,1)$ with Lie algebra
\[
\frak{g}=\frak{u}(2,1)=\bigg\{\begin{bmatrix} X & Y\\ Y^* & Z \end{bmatrix}: X+X^*= Z+Z^*=0\bigg\},
\]
where $X$ and $Y$ are $2\times 2$ and $2\times 1$ blocks respectively. Take a maximal compact subgroup $K=U(2)\times U(1)$ of $G$. Let $\theta(X)=-X^*$ be the Cartan involution and $\frak{g}=\frak{k}+\frak{p}$ the Cartan decomposition of $\frak{g}$ with respect to $\theta$. Then $\frak{k}=\frak{u}(2)\times\frak{u}(1)$, and we have the following maximal abelian subspace of $\frak{p}$
\begin{equation}\label{a}
\frak{a}=\mathbb{R}H:=\mathbb{R}(X_{13}+X_{31})=\left\{\begin{bmatrix}0 & 0 & t \\ 0 & 0 & 0\\  t & 0 & 0\end{bmatrix}: t\in\mathbb{R}\right\}.
\end{equation}
Let $\Sigma$ be the restricted root system $\Sigma(\frak{g},\frak{a})$. Let $\alpha$ be the linear functional $\frak{a}\to \mathbb{R}$, $t H\mapsto t$. Then we may choose a positive subsystem $\Sigma^+=\{\alpha,2\alpha\}$. We set,
\begin{align*}
 &H_1=\sqrt{-1}(X_{11}+X_{33}),\quad H_2=\sqrt{-1}X_{22},\\
 &E_1=\begin{bmatrix} 0 & 1 & 0\\ -1 & 0 & 1\\ 0 & 1 & 0\end{bmatrix},\quad E_2=\sqrt{-1} \begin{bmatrix}0 & 1 & 0\\ 1 & 0 & -1 \\ 0 & 1 & 0\end{bmatrix},\quad E_3=\sqrt{-1}\begin{bmatrix}1 & 0 & -1\\ 0 & 0 & 0\\ 1 & 0 & -1\end{bmatrix}.
\end{align*}
Then $\frak{g}=\frak{c}(\frak{a})+\sum_{\beta\in\Sigma}\frak{g}_\beta$, where
\begin{equation}\label{rootspace}
\left\{
\begin{aligned}
&\frak{c}(\frak{a})=\textrm{centralizer of }\frak{a}=\mathbb{R}H_1+\mathbb{R}H_2,\\
&\frak{g}_\alpha=\mathbb{R}E_1+\mathbb{R}E_2,\quad\frak{g}_{2\alpha}=\mathbb{R}E_3,\quad \frak{g}_{-\beta}={}^t\frak{g}_\beta.
\end{aligned}\right.
\end{equation}
The Haar measure on $G$ under the decomposition $G= KA^+K$, where $A^+=\exp\overline{\frak{a}^+}$, is given by
\[
dg=D_{A}(a)dk_1dadk_2
\]
where $g=k_1ak_2$, $k_1, k_2\in K=U(2)\times U(1)$, $a\in A^+$, and
\[
D_{A}(a)=\prod_{\beta\in\Sigma^+}\big[\sinh \beta(\log a)\big]^{\dim \frak{g}_{\beta}}.
\]
Let
\[
a_t=\exp(tH)=\begin{bmatrix}\cosh t & 0 & \sinh t \\ 0 & 1 & 0\\ \sinh t & 0 & \cosh t\end{bmatrix},
\]
then from (\ref{rootspace}) it follows that
\begin{equation}\label{da}
D(t):=D_{A}(a_t)=\sinh^2 t\sinh (2t),\quad da_t=dt, \quad t\geq 0.
\end{equation}
The Haar measure on $K$ is normalized such that the total volume equals 1. More explicitly, we shall parameterize $k\in K=U(2)\times U(1)$ and the measure $dk$ by
\begin{equation}\label{k}
\left\{
\begin{aligned}
&k=k(\zeta,\theta,\xi,\eta,\gamma)=\textrm{diag}(1,\zeta,1)\kappa(\theta)\textrm{diag}(\xi,\eta,\gamma),\\
& \kappa(\theta):=\begin{bmatrix}\cos\theta & \sin\theta & 0\\ -\sin\theta & \cos\theta & 0\\ 0 & 0 & 1\end{bmatrix},\\
&dk=\sin 2\theta d\theta d\zeta d\xi d\eta d\gamma, \quad \theta\in \left[0,\frac{\pi}{2}\right],\quad \zeta, \xi, \eta, \gamma\in U(1).
\end{aligned}\right.
\end{equation}

\subsection{Explicit joint harmonics} We shall carefully choose certain nonzero vector $\phi\in\mathscr{H}_{\sigma, \sigma'}$ such that our later computation of zeta integrals can be handled effectively. Since $\sigma$ and $\sigma'$ determine each other, it suffices to find some $\phi\in \mathscr{H}$ which is of $\widetilde{K}$-type $\sigma$. Let $\Lambda=(\Lambda_1, \Lambda_2, \Lambda_3)\in \sqrt{-1}\frak{t}^*$ be the highest weight of $\sigma$. In order to be compatible with the notations in the previous section, in the sequel we also write $\sigma$ for its highest weight $\Lambda$ by abuse of notation.
For later use we introduce two parameters
\begin{equation}\label{rs}
r=\Lambda(Z_{12})=\Lambda_1-\Lambda_2,\quad s=\Lambda(Z_{23})=\Lambda_2-\Lambda_3,
\end{equation}
where $Z_{ij}:=X_{ii}-X_{jj}$. The $r, s$ here should not be confused with those previously used in $U(r,s)$.
Then we have
\[
\sigma\cong \big(\rho_r\otimes\det{}^{\Lambda_2}\big)\boxtimes\chi^{\Lambda_3},
\]
where $r=\Lambda_1-\Lambda_2$, $\rho_r= Sym^r$ is the irreducible representation of $U(2)$ with highest weight $(r,0)$, and $\chi$ is the fundamental character
$u\mapsto u$ of $U(1)$. We shall realize $\rho_r$ as
\begin{equation}\label{rhor}
\mathbb{C}[x,y]_r=\{f\in \mathbb{C}[x,y]: f\textrm{ is homogeneous of degree }r\},
\end{equation}
with the group action given by
$\rho_r(g)f(x,y)=f\big((x,y)g\big),$ $g\in U(2)$.

We have the following cases for $G'$ and the pair $(\sigma, \sigma')$.

\

\noindent (A) $G'=U(3,0)=U(3)$. Then $G'$ is compact, thus $\frak{k}'=\frak{g}'$ and $\frak{m}=\frak{g}$. Recall that we have the embeddings $\iota: \frak{g}\hookrightarrow \frak{sp}$, $\iota': \frak{g}'\hookrightarrow\frak{sp}$ given by (\ref{gembed}),
where $J=I_{2,1}\otimes I_3$.
By direct calculations we find
\begin{equation}\label{m1}
\frak{m}^{(0,2)}=\frak{g}_{\mathbb{C}}\cap\frak{p}_-=\bigg\{p_-(A): A=\begin{bmatrix} 0 & Y\\ ^tY & 0\end{bmatrix}\otimes I_3, Y\in M_{2\times 1}(\mathbb{C})\bigg\}, \quad \frak{m}'^{(0,2)}=0.
\end{equation}
Then from (\ref{p-}) and (\ref{m1}) we see that in the Fock model $d\nu\big(\frak{m}^{(0,2)}\big)$ is spanned by
\begin{equation}\label{m1fock}
\frac{\partial^2}{\partial z_{11}\partial z_{31}}+\frac{\partial^2}{\partial z_{12}\partial z_{32}}+
\frac{\partial^2}{\partial z_{13}\partial z_{33}}\quad\textrm{and}\quad\frac{\partial^2}{\partial z_{21}\partial z_{31}}+\frac{\partial^2}{\partial z_{22}\partial z_{32}}+\frac{\partial^2}{\partial z_{23}\partial z_{33}}.
\end{equation}
The corresponding pair of weights are
\[
\sigma=\bigg(\mu_1+\frac{3}{2},\mu_2+\frac{3}{2},-\nu-\frac{3}{2}\bigg),\quad \sigma'=\bigg(\mu_1+\frac{1}{2},\mu_2+\frac{1}{2}, -\nu+\frac{1}{2}\bigg),
\]
where $\mu_1\geq\mu_2\geq 0$, $\nu\geq 0$.
By (\ref{fockk}) it is not hard to check that the following $\phi$ lies in $\mathscr{H}$ and is of $\widetilde{K}$-type $\sigma$
\[
\phi=z_{11}^r \det\begin{bmatrix} z_{11} & z_{12} \\ z_{21} & z_{22} \end{bmatrix}^{\mu_2} z_{33}^\nu=z_{11}^r (z_{11}z_{22}-z_{21}z_{12})^{\mu_2}z_{33}^\nu,
\]
where $r=\mu_1-\mu_2$ in this case.
In fact, modulo a character $\det{}^{\frac{3}{2}}$, the subgroup $U(2)$ acts on $z_{11}z_{22}-z_{21}z_{12}$ by the character $\det$, and acts on the space spanned by $z_{11}^iz_{21}^{r-i}$, $i=0,\ldots, r$ via the representation $\rho_r$.

\

\noindent (B) $G'=U(0,3)=U(3)$. Then $\frak{k}'=\frak{g}'$, $\frak{m}=
\frak{g}$, $J=-I_{2,1}\otimes I_3$. We still have (\ref{m1}) and (\ref{m1fock}).
The corresponding pair of weights are
\[
\sigma=\bigg(-\nu_2-\frac{3}{2}, -\nu_1-\frac{3}{2}, \alpha+\frac{3}{2}\bigg),\quad \sigma'=\bigg(\alpha-\frac{1}{2}, -\nu_2-\frac{1}{2}, -\nu_1-\frac{1}{2}\bigg)
\]
where  $\nu_1\geq\nu_2\geq 0$, $\alpha\geq 0$. Similar to case (A), we may choose
\[
\phi=z_{23}^r (z_{12}z_{23}-z_{22}z_{13})^{\nu_2} z_{31}^{\alpha},\quad r=\nu_1-\nu_2.
\]

\

\noindent (C) $G'=U(2,1)$. Then $J=I_{2,1}\otimes I_{2,1}$. We have
\[
\frak{m}^{(0,2)}=\bigg\{p_-(A): A=\begin{bmatrix} 0 & 0 & A_1\\ 0 & 0 & A_2\\
{}^tA_1 & {}^tA_2 & 0\end{bmatrix}, A_i=\begin{bmatrix} x_i I_2 & 0\\ 0 & y_i\end{bmatrix}, x_i, y_i\in\mathbb{C}, i=1,2\bigg\}
\]
thus $d\nu\big(\frak{m}^{(0,2)}\big)$ is spanned by
\[
\frac{\partial^2}{\partial z_{11}\partial z_{31}}+\frac{\partial^2}{\partial z_{12}\partial z_{32}},\quad \frac{\partial^2}{\partial z_{13}\partial z_{33}},\quad
\frac{\partial^2}{\partial z_{21}\partial z_{31}}+\frac{\partial^2}{\partial z_{22}\partial z_{32}},\quad \frac{\partial^2}{\partial z_{23}\partial z_{33}},
\]
and
\[
\frak{m}'^{(0,2)}=\bigg\{p_-(A): A=\begin{bmatrix} A_1 & 0 & 0\\ 0 & A_1 & 0\\ 0 & 0 & A_2\end{bmatrix}, A_i
=\begin{bmatrix} 0 & Y_i \\ ^tY_i & 0 \end{bmatrix}, Y_i\in M_{2\times 1}(\mathbb{C}), i=1,2\bigg\}
\]
thus $d\nu\big(\frak{m}'^{(0,2)}\big)$ is spanned by
\[
\frac{\partial^2}{\partial z_{11}\partial z_{13}}+\frac{\partial^2}{\partial z_{21}\partial z_{23}},\quad
\frac{\partial^2}{\partial z_{12}\partial z_{13}}+\frac{\partial^2}{\partial z_{22}\partial z_{23}},\quad \frac{\partial^2}{\partial z_{31}\partial z_{33}}, \quad \frac{\partial^2}{\partial z_{32}\partial z_{33}}.
\]
We have two cases for the corresponding pairs of weights:
\begin{equation}\label{C1}
\sigma=\sigma'=\bigg(\mu_1+\frac{1}{2}, \mu_2+\frac{1}{2}, \alpha-\frac{1}{2}\bigg)\tag{C1}
\end{equation}
where $\mu_1\geq \mu_2\geq 0, \alpha\geq 0$, for which we may choose
\[
\phi=z_{11}^r (z_{11}z_{22}-z_{21}z_{12})^{\mu_2}z_{33}^\alpha,\quad r=\mu_1-\mu_2.
\]
\begin{equation}\label{C2}
\sigma=\bigg(\mu+\frac{1}{2}, -\nu+\frac{1}{2}, -\beta-\frac{1}{2}\bigg), \quad \sigma'=\bigg(\mu+\frac{1}{2}, -\beta+\frac{1}{2}, -\nu-\frac{1}{2}\bigg)\tag{C2}
\end{equation}
where $\mu, \nu, \beta\geq 0$. We claim that the vector
\[
\phi=z_{11}^\mu z_{23}^\nu z_{32}^\beta
\]
lies in $\mathscr{H}_{\sigma,\sigma'}$. Clearly $\phi\in\mathscr{H}$. Again modulo a character $\det{}^\frac{1}{2}$, consider the irreducible $U(2)$-modules
$V_{\lambda_1}$, $V_{\lambda_2}$ which are generated by the highest weight vectors $v_{\lambda_1}=z_{11}^\mu$, $v_{\lambda_2}=z_{23}^\nu$ respectively, where $\lambda_1=(\mu, 0)$, $\lambda_2=(0,-\nu)$. Then $\lambda_1+\lambda_2=(\mu,-\nu)$. Under the projection
\[
V_{\lambda_1}\otimes V_{\lambda_2}\to V_{\lambda_1+\lambda_2},
\]
the image of the vector $v_{\lambda_1}\otimes v_{\lambda_2} =z_{11}^\mu z_{23}^\nu$ is a multiple of the highest weight vector $v_{\lambda_1+\lambda_2}$, which generates
$V_{\lambda_1+\lambda_2}$. This proves that $\phi$ is of $\widetilde{K}$-type $\sigma$.

\

\noindent (D) $G'=U(1,2)$. Then $J=I_{2,1}\otimes I_{1,2}$. We have
\[
\frak{m}^{(0,2)}=\bigg\{p_-(A): A=\begin{bmatrix} 0 & 0 & A_1\\ 0 & 0 & A_2\\
{}^tA_1 & {}^tA_2 & 0\end{bmatrix}, A_i=\begin{bmatrix} x_i  & 0\\ 0 & y_i I_2\end{bmatrix}, x_i, y_i\in\mathbb{C}, i=1,2\bigg\}
\]
thus $d\nu\big(\frak{m}^{(0,2)}\big)$ is spanned by
\[
\frac{\partial^2}{\partial z_{11}\partial z_{31}},\quad \frac{\partial^2}{\partial z_{12}\partial z_{32}}+\frac{\partial^2}{\partial z_{13}\partial z_{33}},\quad \frac{\partial^2}{\partial z_{21}\partial z_{31}},\quad \frac{\partial^2}{\partial z_{22}\partial z_{32}}+\frac{\partial^2}{\partial z_{23}\partial z_{33}},
\]
and
\[
\frak{m}'^{(0,2)}=\bigg\{p_-(A): A=\begin{bmatrix} A_1 & 0 & 0\\ 0 & A_1 & 0\\ 0 & 0 & A_2\end{bmatrix}, A_i
=\begin{bmatrix} 0 & Y_i \\ ^tY_i & 0 \end{bmatrix}, Y_i\in M_{1\times 2}(\mathbb{C}), i=1,2\bigg\}
\]
thus $d\nu\big(\frak{m}'^{(0,2)}\big)$ is spanned by
\[
\frac{\partial^2}{\partial z_{11}\partial z_{12}}+\frac{\partial^2}{\partial z_{21}\partial z_{22}},\quad
\frac{\partial^2}{\partial z_{11}\partial z_{13}}+\frac{\partial^2}{\partial z_{21}\partial z_{23}},\quad \frac{\partial^2}{\partial z_{31}\partial z_{32}},\quad \frac{\partial^2}{\partial z_{31}\partial z_{33}}.
\]
We have two cases for the corresponding pairs of weights:
\begin{equation}\label{D1}
 \sigma=\bigg(-\nu_2-\frac{1}{2},-\nu_1-\frac{1}{2},-\beta+\frac{1}{2}\bigg),\quad \sigma'=\bigg(-\beta+\frac{1}{2}, -\nu_2-\frac{1}{2}, -\nu_1-\frac{1}{2}\bigg) \tag{D1}
\end{equation}
where $\nu_1\geq \nu_2\geq 0$, $\beta\geq 0$, for which we may choose
\[
\phi=z_{23}^r (z_{12}z_{23}-z_{22}z_{13})^{\nu_2} z_{31}^{\beta},\quad r=\nu_1-\nu_2.
\]
\begin{equation}\label{D2}
\sigma=\bigg(\mu-\frac{1}{2}, -\nu-\frac{1}{2}, \alpha+\frac{1}{2}\bigg),\quad \sigma'=\bigg(\mu+\frac{1}{2}, \alpha-\frac{1}{2}, -\nu-\frac{1}{2}\bigg)\tag{D2}
\end{equation}
where $\mu, \nu, \alpha\geq 0$. Similarly with case (C2), we may choose
\[
\phi=z_{11}^\mu z_{23}^\nu z_{32}^\alpha.
\]

\subsection{Matrix coefficients for the joint harmonics}

For the cases (A)$\sim$(D), we shall calculate the matrix coefficients $( \omega(g)\phi,\phi)$ for $g\in G$ and $\phi$ being chosen as in Section 3.2. We first give some notations and formulas that are used in our computation.

Write $g$ according to the $KAK$ decomposition as $g= k'^{-1}a k$, then
\[
( \omega(g)\phi,\phi)=(\omega(ak)\phi, \omega(k')\phi).
\]
We parametrize $k=k(\zeta,\theta,\xi,\eta,\gamma)$, $k'=k(\zeta',\theta',\xi',\eta',\gamma')$ as in (\ref{k}). For $a_t\in A^+$, we find that its image in $Sp$ via the embedding $\iota: G\hookrightarrow Sp$ is
\[
\iota(a_t)=\begin{bmatrix} a_t\otimes I & 0 \\ 0 & a_{-t}\otimes I \end{bmatrix},
\]
hence its image  $a_t^c$ in $Sp^c$ reads
\[
a^c_t=\begin{bmatrix} P & Q \\ Q & P\end{bmatrix}, \quad P=\begin{bmatrix}\cosh t & 0 & 0\\ 0 & 1 & 0\\ 0 & 0 & \cosh t\end{bmatrix}\otimes I,
\quad Q=\begin{bmatrix}0 & 0 & \sinh t\\ 0 & 0 & 0\\ \sinh t & 0 & 0\end{bmatrix}\otimes I.
\]
Write $z=[z_{ij}]_{i,j=1,2,3}$ with ordering of entries as before, and similarly $w=[w_{ij}]$. Then
\[
\left\{
\begin{aligned}
&QP^{-1}=P^{-1}Q=\begin{bmatrix} 0 & 0 & \tanh t \\ 0 & 0 & 0\\ \tanh t & 0 & 0 \end{bmatrix}\otimes I,\\
&{}^tzQP^{-1}z=2\tanh t(z_{11}z_{31}+z_{12}z_{32}+z_{13}z_{33}),\\
& {}^t\bar{w}P^{-1}z=\sum^3_{i=1}(\cosh^{-1} t(z_{1i}\bar{w}_{1i}+z_{3i}\bar{w}_{3i}) + z_{2i}\bar{w}_{2i}),\\
& {}^t\bar{w}P^{-1}Q\bar{w}=2\tanh t(\bar{w}_{11}\bar{w}_{31}+\bar{w}_{12}\bar{w}_{32}+\bar{w}_{13}\bar{w}_{33}).
\end{aligned}
\right.
\]
We will frequently use the following formulas
\begin{equation}\label{formula}
\left\{
\begin{aligned}
&\int_{\mathbb{C}}z^l e^{\pi \lambda\bar{z}-\pi|z|^2}dz=\lambda^l,\quad \int_{\mathbb{C}}\bar{z}^{l+1}e^{\pi \lambda\bar{z}-\pi|z|^2}dz=0,\\
&\int_{\mathbb{C}}|z|^{2l}e^{\pi\lambda z -\pi|z|^2}dz=\int_{\mathbb{C}}|z|^{2l}e^{\pi\lambda \bar{z} -\pi|z|^2}dz=\frac{l!}{\pi^l},
\end{aligned}\right.
\end{equation}
where $l\geq 0$, $\lambda$ is a constant.

\

Let us calculate case (A) for example, and only present the final formulas for other cases. As in Section 3.2, now we have
 $\phi=z_{11}^r (z_{11}z_{22}-z_{21}z_{12})^{\mu_2}z_{33}^\nu$, $\sigma=(\mu_1+\frac{3}{2},\mu_2+\frac{3}{2}, -\nu-\frac{3}{2})$
 with $r=\mu_1-\mu_2$. We first calculate $\omega(k)\phi$. Then by (\ref{taction})
\[
\omega\big[\textrm{diag}(\xi,\eta,\gamma)\big]\phi= \xi^{\mu_1+\frac{3}{2}}\eta^{\mu_2+\frac{3}{2}}\bar{\gamma}^{\nu+\frac{3}{2}}\phi.
\]
Using (\ref{fockk}) we find that
\[
\omega\big[\kappa(\theta)\big]\phi=\nu\begin{bmatrix} \kappa(\theta)\otimes I & 0 \\ 0 & \kappa(\theta)\otimes I \end{bmatrix}\phi
=(z_{11}\cos\theta -z_{21}\sin\theta )^r(z_{11}z_{22}-z_{21}z_{12})^{\mu_2}z_{33}^\nu.
\]
Again from (\ref{taction}) it follows
\[
\omega\big[\textrm{diag}(1,\zeta,1)\kappa(\theta)\big]\phi=\zeta^{\mu_2+\frac{3}{2}}(z_{11}\cos\theta -z_{21}\zeta\sin\theta )^r(z_{11}z_{22}-z_{21}z_{12})^{\mu_2}z_{33}^\nu.
\]
Combing above equations, we have derived the formula for $\omega(k)\phi$. We have a similar formula for $\omega(k')\phi$. It remains to give the action of $\omega(a_t)$ using (\ref{fockaction}), i.e. calculate
\begin{align*}
&\nu(a_t^c)\big[(z_{11}\cos\theta -z_{21}\zeta\sin\theta )^r(z_{11}z_{22}-z_{21}z_{12} )^{\mu_2}z_{33}^\nu\big]\\
=&\int_{\mathbb{C}^9}K_{a_t^c}(z,\bar{w})(w_{11}\cos\theta -w_{21}\zeta\sin\theta )^r(w_{11}w_{22}-w_{21}w_{12} )^{\mu_2}w_{33}^\nu e^{-\pi|w|^2}dw.
\end{align*}
Applying (\ref{formula}) repeatedly, we find that
\begin{align*}
\omega(a_t k)\phi=& ~\xi^{\mu_1+\frac{3}{2}}(\eta\zeta)^{\mu_2+\frac{3}{2}}\bar{\gamma}^{\nu+\frac{3}{2}}(\cosh t)^{-\mu_2-\nu-3}e^{\pi \tanh t (z_{11}z_{31}+z_{12}z_{32}+z_{13}z_{33})}\\
&\times (z_{11}\cosh^{-1}t\cos\theta -z_{21}\zeta\sin\theta)^r(z_{11}z_{22}-z_{21}z_{12})^{\mu_2}z_{33}^\nu.
\end{align*}
Then the matrix coefficient
$(\omega(g)\phi,\phi)$ equals
\begin{align*}
(\omega(a_tk)\phi, \omega(k')\phi) =&\int_{\mathbb{C}^9}\omega(a_tk)\phi\cdot\overline{\omega(k')\phi}\cdot e^{-\pi |z|^2}dz\\
=&~(\xi\bar{\xi}')^{\mu_1+\frac{3}{2}}(\eta\bar{\eta}'\zeta\bar{\zeta}')^{\mu_2+\frac{3}{2}}(\bar{\gamma}\gamma')^{\nu+\frac{3}{2}}(\cosh t)^{-\mu_2-\nu-3} J_1 J_2,
\end{align*}
where
\begin{align*}
J_1=& \int_{\mathbb{C}^6} (z_{11}\cosh^{-1}t\cos\theta -z_{21}\zeta\sin\theta)^r\overline{ (z_{11}\cos\theta' -z_{21}\zeta'\sin\theta')}^r\\
& \times\big|z_{11}z_{22}-z_{21}z_{12}\big|^{2\mu_2}e^{-\pi(|Z_1|^2+ |Z_2|^2)}dZ_1dZ_2\\
=&\sum^r_{i=0}{r\choose i}^2(\zeta\bar{\zeta}')^i(\cosh t)^{i-r}(\sin\theta\sin\theta')^i(\cos\theta\cos\theta')^{r-i}\sum^{\mu_2}_{j=0}{\mu_2 \choose j}^2\| z_{11}^{j+r-i}z_{21}^{\mu_2-j+i}z_{22}^jz_{12}^{\mu_2-j}\|^2\\
=&\frac{\mu_2!}{\pi^{\mu_1+\mu_2}}\sum^r_{i=0}{r\choose i}^2(\zeta\bar{\zeta}')^i(\cosh t)^{i-r}(\sin\theta\sin\theta')^i(\cos\theta\cos\theta')^{r-i}\sum^{\mu_2}_{j=0}{\mu_2 \choose j}(j+r-i)!(\mu_2-j+i)!\\
J_2=& \int_{\mathbb{C}^3} |z_{33}|^{2\nu} e^{\pi \tanh t(z_{11}z_{31}+z_{12}z_{32}+ z_{13}z_{33})- \pi |Z_3|^2 } dZ_3= \frac{\nu !}{\pi^\nu}.
\end{align*}
In the above $Z_j=(z_{1j}, z_{2j}, z_{3j})$, $j=1,2,3$, and we have used (\ref{formula}) for $J_2$. The following elementary combinatorial lemma enables us
 to further simply $J_1$.

 \begin{lemma} \label{comb} Fix $\mu, r\geq 0$. Then for $i=0,\ldots, r$,
 \[
 {r\choose i}\sum^\mu_{j=0}{\mu\choose j}(j+r-i)!(\mu-j+i)!=\frac{(\mu+r+1)!}{r+1}.
 \]
\end{lemma}

\begin{proof}
Compare the coefficients of $x^\mu$ on both sides of the equation
\[
(1+x)^{i-r-1}(1+x)^{-i-1}=(1+x)^{-r-2}.
\]
Then we get
\[
\sum^\mu_{j=0}{i-r-1\choose j} {-i-1\choose \mu-j}= {-r-2\choose \mu},
\]
which implies
\[
\sum^\mu_{j=0}\frac{(j+r-i)!}{(r-i)!j!} \frac{(\mu-j+i)!}{i!(\mu-j)!}=\frac{(\mu+r+1)!}{(r+1)!\mu!}.
\]
The lemma follows immediately.
\end{proof}

It turns out that this lemma is also useful in the computation of zeta integrals for some cases, see Theorem C2 in Section 5. Now notice that in particular we have
\[
\|\phi\|^2=\frac{\mu_2!\nu!}{\pi^{\mu_1+\mu_2+\nu}}\sum^{\mu_2}_{j=0}{\mu_2\choose j}(j+r)!(\mu_2-j)!=\frac{\mu_2!\nu!}{\pi^{\mu_1+\mu_2+\nu}}\frac{(\mu_1+1)!}{r+1}.
\]
Combining with above lemma, in summary for case (A) we obtain
\begin{align}
\label{focka}
(\omega(g)\phi, \phi)=&~\|\phi\|^2 (\xi\bar{\xi}')^{\mu_1+\frac{3}{2}}(\eta\bar{\eta}'\zeta\bar{\zeta}')^{\mu_2+\frac{3}{2}}(\bar{\gamma}\gamma')^{\nu+\frac{3}{2}}(\cosh t)^{-\mu_2-\nu-3} \\
&\times \sum^r_{i=0}{r\choose i}(\zeta\bar{\zeta}')^i(\cosh t)^{i-r}(\sin\theta\sin\theta')^i(\cos\theta\cos\theta')^{r-i}\nonumber\\
=&~\|\phi\|^2 (\xi\bar{\xi}')^{\mu_1+\frac{3}{2}}(\eta\bar{\eta}'\zeta\bar{\zeta}')^{\mu_2+\frac{3}{2}}(\bar{\gamma}\gamma')^{\nu+\frac{3}{2}}(\cosh t)^{-\mu_2-\nu-3} \nonumber\\
&\times (\cosh^{-1}t\cos\theta\cos\theta'+\zeta\bar{\zeta}'\sin\theta\sin\theta')^r.\nonumber
\end{align}

Similar calculations give us the following

\noindent (B) $\phi=z_{23}^r (z_{12}z_{23}-z_{22}z_{13})^{\nu_2} z_{31}^{\alpha}$, $\sigma=(-\nu_2-\frac{3}{2}, -\nu_1-\frac{3}{2}, \alpha+\frac{3}{2})$
with $r=\nu_1-\nu_2$. Then
\begin{align}
\label{fockb}(\omega(g)\phi,\phi)=&~\|\phi\|^2 (\bar{\xi}\xi')^{\nu_2+\frac{3}{2}}(\bar{\eta}\eta'\bar{\zeta}\zeta')^{\nu_1+\frac{3}{2}}(\gamma\bar{\gamma}')^{\alpha+\frac{3}{2}}(\cosh t)^{-\nu_2-\alpha-3} \\
&\times ( \cos \theta \cos \theta' + \cosh^{-1}t\zeta\bar{\zeta}'\sin \theta \sin \theta')^r.\nonumber
\end{align}

\noindent (C1) $z_{11}^r (z_{11}z_{22}-z_{21}z_{12})^{\mu_2}z_{33}^\alpha$, $\sigma=(\mu_1+\frac{1}{2},\mu_2+\frac{1}{2},\alpha-\frac{1}{2})$
with $r=\mu_1-\mu_2$. Then
\begin{align}
\label{fockc1}(\omega(g)\phi, \phi)=&~\|\phi\|^2 (\xi\bar{\xi}')^{\mu_1+\frac{1}{2}}(\eta\bar{\eta}'\zeta\bar{\zeta}')^{\mu_2+\frac{1}{2}}(\gamma\bar{\gamma}')^{\alpha-\frac{1}{2}}(\cosh t)^{-\mu_2-\alpha-3} \\
&\times (\cosh^{-1}t \cos \theta \cos \theta' + \zeta\bar{\zeta}'\sin \theta \sin \theta')^r.\nonumber
\end{align}

\noindent (C2) $\phi=z_{11}^\mu z_{23}^\nu z_{32}^\beta$, $\sigma=(\mu+\frac{1}{2}, -\nu+\frac{1}{2}, -\beta-\frac{1}{2})$. Then
\begin{align}
\label{fockc2}&(\omega(g)\phi,\phi)=\|\phi\|^2 (\xi\bar{\xi}')^{\mu+\frac{1}{2}}(\bar{\eta}\eta')^{\nu-\frac{1}{2}}(\bar{\gamma}\gamma')^{\beta+\frac{1}{2}}(\zeta\bar{\zeta}')^{\frac{1}{2}}(\cosh t)^{-\beta-3} \\
&\quad\times (\cosh^{-1} t \cos \theta \cos \theta' + \zeta\bar{\zeta}'\sin \theta \sin \theta')^\mu(\cosh^{-1}t \sin\theta \sin \theta' + \bar{\zeta}\zeta'\cos\theta \cos\theta' )^\nu.\nonumber
\end{align}

\noindent (D1) $\phi=z_{23}^r (z_{12}z_{23}-z_{22}z_{13})^{\nu_2} z_{31}^{\beta}$, $\sigma=(-\nu_2-\frac{1}{2}, -\nu_1-\frac{1}{2}, -\beta+\frac{1}{2})$
with $r=\nu_1-\nu_2$. Then
\begin{align}
\label{fockd1}(\omega(g)\phi,\phi)=&~\|\phi\|^2 (\bar{\xi}\xi')^{\nu_2+\frac{1}{2}}(\bar{\eta}\eta'\bar{\zeta}\zeta')^{\nu_1+\frac{1}{2}}(\bar{\gamma}\gamma')^{\beta-\frac{1}{2}}(\cosh t)^{-\nu_2-\beta-3} \\
&\times ( \cos \theta \cos \theta' +\cosh^{-1}t \zeta\bar{\zeta}'\sin \theta \sin \theta')^r. \nonumber
\end{align}

\noindent (D2) $\phi=z_{11}^\mu z_{23}^\nu z_{32}^\alpha$, $\sigma=(\mu-\frac{1}{2}, -\nu-\frac{1}{2}, \alpha+\frac{1}{2})$. Then
\begin{align}
\label{fockd2}&(\omega(g)\phi,\phi)=\|\phi\|^2 (\xi\bar{\xi}')^{\mu-\frac{1}{2}}(\bar{\eta}\eta')^{\nu+\frac{1}{2}}(\gamma\bar{\gamma}')^{\alpha+\frac{1}{2}}(\bar{\zeta}\zeta')^{\frac{1}{2}}(\cosh t)^{-\alpha-3} \\
&\quad\times (\cosh^{-1} t \cos \theta \cos \theta' + \zeta\bar{\zeta}'\sin \theta \sin \theta')^\mu(\cosh^{-1}t \sin\theta \sin \theta' + \bar{\zeta}\zeta'\cos\theta \cos\theta' )^\nu.\nonumber
\end{align}

\section{Matrix coefficients of the discrete series of $U(2,1)$}

Let $(\pi_\lambda, H_\lambda)$ be a discrete series representation of $U(2,1)$ with Harish-Chandra parameter $\lambda$. The aim of this section is to calculate explicitly the matrix coefficients of $\pi_\lambda$. We shall follow the method in \cite{HaKO}. We first remark that it is essentially equivalent to work with $SU(2,1)$ instead of $U(2,1)$, which we shall do for convenience. In particular the matrix coefficients for these two groups have the same restriction to $A$. Until subsection 4.5, we let $G=SU(2,1)$ and for convenience we also adopt the same notations $\frak{g}$, $\frak{k}$, $\frak{p}$ etc. for $SU(2,1)$.

\subsection{Discrete series and $K$-types}

Take a maximal compact subgroup $K=S(U(2)\times U(1))\cong U(2)$, with the obvious isomorphism given by $(g,\det g^{-1})\mapsto g$.  Let $\frak{g}=\frak{k}+\frak{p}$
be the Cartan decomposition. The maximal abelian subspace $\frak{a}$ of $\frak{p}$, the root systems $\Sigma$, $\Sigma^+$, and the root subspaces $\frak{g}_\beta$ are the same as in Section 3.1, while
 $\frak{c}(\frak{a})$ is now spanned by $H_0=\sqrt{-1}\textrm{diag}(1,-2,1)$. Denote by $\frak{n}$ the nilpotent subalgebra $\sum_{\beta\in\Sigma^+}\frak{g}_\beta$. Then we have the Iwasawa decomposition $G=NAK$ with $N=\exp\frak{n}$.

The discrete series of $SU(2,1)$ can be parametrized as follows. Take the compact Cartan subalgebra $\frak{t}\subset\frak{k}$ given by
\[
\frak{t}=\mathbb{R}\sqrt{-1}Z_{12} +\mathbb{R} \sqrt{-1}Z_{23}
\]
with $Z_{ij}=X_{ii}-X_{jj}$. The absolute root system of type $A_2$ is given by
\[
\Delta=\Delta(\frak{g}_{\mathbb{C}},\frak{t}_{\mathbb{C}})=\{e_i-e_j, 1\leq i\neq j\leq 3\}.
\]
Fix the compact positive root $\Delta_c^+=\{e_1-e_2\}$. The Weyl group $W=W(\frak{g}_{\mathbb{C}}, \frak{t}_{\mathbb{C}})\cong S_3$, and the compact Weyl group
is $W_c\cong S_2$. There are three positive subsystems containing $\Delta_c^+$,
\[
\left\{
\begin{aligned}
& \Delta_{I}^+=\{e_1-e_2, e_2-e_3, e_1-e_3\},\\
& \Delta_{II}^+=\{e_3-e_1, e_1-e_2, e_3-e_2\}=w_{ II}\Delta_{ I}^+,\\
& \Delta_{III}^+=\{e_1-e_3, e_3-e_2, e_1-e_2\}=w_{ III} \Delta_{ I}^+,
\end{aligned}\right.
\]
where $w_J\in W$ are given by $w_{II}=(132)$, $w_{III}= (23)$. The space of Harish-Chandra parameters $\Xi_c$ is given by
\[
\Xi_c=\{\lambda\in \sqrt{-1}\frak{t}^*:  \lambda\textrm{ is }\Delta\textrm{-regular, }K\textrm{-analytically integral and }\Delta_c^+\textrm{-dominant}\}.
\]
Denote by $\Delta_{J,n}^+$ the non-compact roots in $\Delta^+_J$.

Put $\Xi_J=\{\lambda\in \Xi_c: \lambda\textrm{ is }\Delta_J^+\textrm{-dominant}\}$. Then $\Xi_c=\Xi_I\cup \Xi_{II}\cup \Xi_{III}$. Let
\[
\rho_{J}=\frac{1}{2}\sum_{\beta\in \Delta^+_J}\beta,\quad \rho_c=\frac{1}{2}\sum_{\beta\in \Delta^+_c}\beta,\quad \rho_{J,n}=\frac{1}{2}\sum_{\beta\in \Delta^+_{J,n}}\beta
\]
be the half sum of positive roots, compact positive roots and non-compact positive roots respectively. For $\lambda\in \Xi_J$, let
$(\pi_\lambda, H_\lambda)$ be the corresponding discrete series representation. Then the Blatter parameter of $\pi_\lambda$ is
\[
\Lambda=\lambda+\rho_{J}-2\rho_c.
\]
Let $\tau=(\tau_\Lambda, V_\Lambda)$ be the minimal $K$-type of $\pi_\lambda$. Write $\Lambda=(\Lambda_1, \Lambda_2, \Lambda_3)\in \sqrt{-1}\frak{t}^*$, and let $r, s$ be the parameters given by (\ref{rs}). Then under the isomorphism $K\cong U(2)$, we have
\[
\tau_\Lambda\cong \rho_r\otimes \det{}^s,
\]
where $\rho_r$ is realized as (\ref{rhor}). Let $f_i=x^i y^{r-i}$, $i=0,1,\ldots, r$ be the standard basis of $\rho_r$, and denote by $f^\Lambda_i$ the corresponding basis of $V_\Lambda$.
The action of $\frak{k}$ is then given by
\begin{equation}\label{kaction}
\left\{
\begin{aligned}
&\tau(Z_{12})f^\Lambda_i= (2i-r)f^\Lambda_i, \quad \tau(Z_{23})f^\Lambda_i=(r+s-i)f^\Lambda_i, \\
&\tau(X_{12})f^\Lambda_i= (r-i)f^\Lambda_{i+1}, \quad \tau(X_{21})f^\Lambda_i= i f^{\Lambda}_{i-1}.
\end{aligned}\right.
\end{equation}
For convenience we shall also write $(\tau_{r,s}, V_{r,s})$ for $(\tau_\Lambda, V_\Lambda)$, and $f^{r,s}_i$ for $f^\Lambda_i$. When $s=0$ we further write $f^r_i$ for $f^{r,0}_i$.

We have the decomposition (cf.  \cite[Lemma 4.1]{Y})
\[
\rho_r\otimes \rho_1\cong \rho_{r+1}\oplus \tau_{r-1,1} =\rho_{r+1}\oplus [\rho_{r-1}\otimes \det(\cdot)].
\]
Let $P^\pm_r$ be the projections from $\rho_r\otimes \rho_1$ to $\rho_{r+1}$ and $\tau_{r-1,1} $. Then with above notations we have (cf. \cite[Lemma 3.11]{Ha})

\begin{lemma}\label{proj}
Let $0\leq i\leq r$ and $e=0,1$. Then
\[
\left\{
\begin{aligned} & P^+_r(f^r_i\otimes f^1_e)=f^{r+1}_{i+e},\\
& P^-_r(f^r_i\otimes f^1_e)=(r\delta_{e1}-i)f^{r-1,1}_{i-1+e}.
\end{aligned}
\right.
\]
\end{lemma}

The adjoint representation Ad of $K$ on $\frak{p}_{\mathbb{C}}$ is decomposed into a direct sum of two irreducible subrepresentations $\frak{p}_{\mathbb{C}}=\frak{p}_++
\frak{p}_-$, where
\[
\frak{p}_+=\mathbb{C}X_{13}+\mathbb{C}X_{23},\quad \frak{p}_-={}^t\frak{p}_+.
\]
For later use we fix the $K$-isomorphisms
\begin{equation}\label{ad}
\left\{
\begin{aligned}
&\textrm{Ad}_+=\textrm{Ad}\big|_{\frak{p_+}}\cong \tau_{1,1},\quad (X_{13}, X_{23})\mapsto (f^{1,1}_1, f^{1,1}_0)\\
&\textrm{Ad}_-=\textrm{Ad}\big|_{\frak{p_-}}\cong \tau_{1,-2},\quad(X_{31}, X_{32})\mapsto (f^{1,-2}_0, - f^{1,-2}_1).
\end{aligned}\right.
\end{equation}
The irreducible decomposition of $\frak{k}_{\mathbb{C}}$-module $V_\Lambda\otimes \frak{p}_{\mathbb{C}}$ is therefore given by
 \begin{align*}
& V_\Lambda\otimes\frak{p}_{\mathbb{C}} =V_\Lambda\otimes\frak{p}_+\oplus  V_\Lambda\otimes\frak{p}_-,\\
& V_\Lambda\otimes \frak{p}_+\cong V_{r+1, s+1}\oplus V_{r-1,s+2},\\
& V_\Lambda\otimes\frak{p}_-\cong V_{r+1,s-2}\oplus V_{r-1,s-1}.
  \end{align*}
Using $P_r^\pm$ in Lemma \ref{proj}, we obtain the projectors $P_\tau^\pm$ defined on $V_\Lambda\otimes\frak{p}_{\mathbb{C}}$:
\begin{equation}\label{projector}
\left\{
\begin{aligned}
&P_\tau^+: V_\Lambda\otimes \frak{p}_+\to V_{r+1,s+1},\quad V_\Lambda\otimes \frak{p}_-\to V_{r+1,s-2},\\
&P_\tau^-: V_\Lambda\otimes \frak{p}_+\to V_{r-1,s+2},\quad V_\Lambda\otimes \frak{p}_-\to V_{r-1,s-1}.
\end{aligned}
\right.
\end{equation}
The explicit formulas for $P_\tau^\pm$ in terms of standard basis can be derived easily from Lemma \ref{proj} and (\ref{ad}).

The contragredient representation $\tau^\vee=(\tau_\Lambda^\vee, V^\vee_\Lambda)$ of $(\tau_\Lambda, V_\Lambda)$ is isomorphic to
$(\tau_{\Lambda^\vee}, V_{\Lambda^\vee})$, where $\Lambda^\vee=(-\Lambda_2,-\Lambda_1, -\Lambda_3)$. Then $\tau_{\Lambda^\vee}=\tau_{r, -r-s}$. Let $h_i^\Lambda \in V_\Lambda^\vee$ be the dual basis of $f^\Lambda_i$, i.e. $\langle f^\Lambda_i, h^\Lambda_j\rangle=\delta_{ij}$. Then the correspondence between the basis,
\begin{equation}\label{iso}
h^\Lambda_i\mapsto (-1)^i{r \choose i}f^{\Lambda^\vee}_{r-i}
\end{equation}
determines the  isomorphism $\tau^\vee\cong \tau_{\Lambda^\vee}$, which is unique up to constant (cf. \cite[Proposition 3.9]{Ha}). In particular one has the paring $\langle f^\Lambda_i, f_{r-i}^{\Lambda^\vee}\rangle=(-1)^i{r\choose i}^{-1}$.

\subsection{$(\tau,\tau^\vee)$-matrix coefficients and Schmid operator}

 We shall identify the representation spaces $V_\Lambda$, $V_\Lambda^\vee$ with their unique images in $H_\lambda$, $H_\lambda^\vee$ respectively.
 Then the matrix coefficient of $\pi_\lambda$ is given by
 \[
\langle \pi_\lambda(g)u, u^\vee\rangle
 \]
 for $u\in V_\Lambda$, $u^\vee\in V_\Lambda^\vee$. For convenience we consider a vector-valued function
 \[
 \phi_\lambda(g)=\sum_{v,w}\langle\pi_\lambda(g) v^\vee, w^\vee\rangle w\otimes v,
 \]
 where $\{w\}$ (resp. $\{v\}$) runs over a basis of $V_\Lambda$ (resp. $V_{\Lambda^\vee}$) and $\{w^\vee\}$ (resp. $\{v^\vee\}$) is its dual basis. Note that we shall identify $V^\vee_{\Lambda^\vee}$ with $V_\Lambda$ whenever it is necessary. Then $\phi_\lambda$ is independent of the choice of basis, and belongs to the function space
 \[
 C^\infty_{\tau, \tau^\vee}(K\backslash G/ K)=\{\phi: G\to V_{\Lambda}\otimes V_{\Lambda^\vee} : \phi(k_1gk_2)=\tau(k_1)\otimes \tau^\vee(k_2^{-1})\cdot\phi(g), k_i\in K\}.
 \]
 The spaces $C^\infty_\tau(K\backslash G)$ and $C^\infty_{\tau^\vee}(G/K)$ are also defined in the obvious manner.

By composing $\phi_\lambda$ with the natural map $V_\Lambda\otimes V_{\Lambda^\vee}\cong V_\Lambda\otimes V_\Lambda^\vee\to \mathbb{C}$, we obtain
a scalar-valued function
\begin{equation}\label{canonical}
\psi_\lambda(g)=\sum_{v,w}\langle\pi_\lambda(g) v^\vee, w^\vee\rangle \langle w, v\rangle \in  C^\infty(G),
\end{equation}
which is of right $K$-type $\tau$ and left $K$-type $\tau^\vee$, and satisfies
\[
\left\{\begin{aligned} &\psi_\lambda(kgk^{-1})=\psi_\lambda(g),\quad g\in G, ~k\in K,\\
& \psi_\lambda(1)=\dim\tau.
\end{aligned}\right.
\]
Let $P_\tau$ be the orthogonal projection to the $\tau$-component of $\pi_\lambda$, then we see that
\[
\psi_\lambda(g)=\textrm{Tr}(P_\tau\pi_\lambda(g)P_\tau),\quad g\in G.
\]
This is the canonical matrix coefficient considered in \cite{HLS}, which differs from that in \cite{F-J} by a factor of
$\dim \tau$ (cf. \cite[Lemma 3.1, Lemma 3.3]{L1}). To calculate  $\phi_\lambda(g)$ explicitly, we take the standard basis $\{f^\Lambda_i\}_{0\leq i\leq r}$, $\{f^{\Lambda^\vee}_j\}_{0\leq j\leq r}$ of $V_\Lambda$,
$V_{\Lambda^\vee}$ respectively. Then we can write $\phi_\lambda(g)$ as
\[
\phi_\lambda(g)=\sum_{0\leq i, j\leq r}c_{ij}(g)f^\Lambda_i\otimes f^{\Lambda^\vee}_j
\]
with the coefficients $c_{ij}(g)=\langle \pi_\lambda(g) h^{\Lambda^\vee}_j, h^\Lambda_i \rangle$.
Due to the Cartan decomposition $G=KAK$, $c_{ij}(g)$ is uniquely determined by its restriction to $A$.

\begin{lemma}
$c_{ij}(a)=0$ unless $i+j=r$.
\end{lemma}

\begin{proof}
The centralizer of $A$ in $K$ is
\[
Z_K(A)=\{m=\textrm{diag}(u,\bar{u}^2, u) : u\in U(1)\}.
\]
Applying spherical properties of $\phi_\lambda\in C^\infty_{\tau,\tau^\vee}(K\backslash G/K)$ to the equation
\[
\phi_\lambda(mam^{-1})=\phi_\lambda(a),\quad a\in A, ~m\in Z_K(A)
\]
yields the assertion.
\end{proof}

This lemma enables us to write
\begin{equation}\label{ci}
\phi_\lambda(a)=\sum^r_{i=0}c_i(a)f^\Lambda_i\otimes f^{\Lambda^\vee}_{r-i},
\end{equation}
where $c_i(a):=c_{i,r-i}(a)$.

Now let us recall the Schmid operator from \cite{S}.  Since the adjoint representation Ad of $K$ on $\frak{p}_{\mathbb{C}}$ has all weights of multiplicity
one, it follows that we have a decomposition
\[
\tau_\Lambda\otimes \textrm{Ad}= \sum_{\beta\in \Delta_n}m_\beta \tau_{\Lambda+\beta}
\]
with each $m_\beta$ equal to $0$ or $1$. In fact we see that each $m_\beta$ equals $1$ in our case. Let $\tau^-_\Lambda$
be the subrepresentation of this tensor product given by
\begin{equation}\label{tau-}
\tau^-_\Lambda=\sum_{\beta\in \Delta_{J,n}^+}\tau_{\Lambda-\beta}.
\end{equation}
Then $\tau^-_\Lambda$ acts in a subspace $V^-_\Lambda$ of $V_\Lambda\otimes \frak{p}_{\mathbb{C}}$, and we let
\begin{equation}\label{pj}
P^J_\tau: V_\Lambda\otimes\frak{p}_{\mathbb{C}}\to V^-_\Lambda = \bigoplus_{\beta\in \Delta_{J,n}^+}V_{\Lambda-\beta}
\end{equation}
be the orthogonal projection, which is $K$-equivariant, i.e.
\[
P^J_\tau \circ (\tau\otimes \mathrm{Ad})(k)=\tau_\Lambda^-(k)\circ P^J_\tau,\quad k\in K.
\]

 The Hermitian form
 \[
 \langle X, Y\rangle =- B(X, \theta \overline{Y})
 \]
 gives a positive definite inner product on $\frak{g}_{\mathbb{C}}$, where $B$ is the Killing form. Let $\{X_i\}$ be an orthonormal basis of $\frak{p}$ with respect
 to this inner product. Define
 \begin{equation}\label{nabla}
 \left\{
 \begin{aligned}
 &\nabla^L\phi(g)=\sum_i  l_{X_i} \phi(g)\otimes X_i,\quad \phi\in C^\infty_\tau(K\backslash G),\\
 &\nabla^R\phi(g)=\sum_i r_{X_i}\phi(g)\otimes X_i,\quad \phi\in C^\infty_{\tau^\vee}(G/K)
 \end{aligned}\right.
\end{equation}
 where $l_X$ (resp. $r_X$) is the left (resp. right) differentiation given by
 \[
 l_X\phi(g)=\frac{d}{dt}\phi\big(\exp(-tX)g\big)_{t=0} \quad \textrm{resp.}\quad r_X\phi(g)=\frac{d}{dt}\phi\big(g\exp(tX)\big)_{t=0}.
 \]
 Obviously $\nabla^L$ (resp. $\nabla^R$) is independent of the choice of basis, and is equivariant with respect to right (resp. left) translation by $G$.

If we write $\lambda=(\lambda_1,\lambda_2,\lambda_3)$ and $\lambda^\vee=(-\lambda_2, -\lambda_1, -\lambda_3)$, then $\pi_{\lambda}^\vee\cong \pi_{\lambda^\vee}$, the discrete series with Harish-Chandra parameter $\lambda^\vee$. Note that $\lambda\in \Xi_J$ implies  $\lambda^\vee\in \Xi_{J^\vee}$, where $(J, J^\vee)=(I, II)$, $(II, I)$ or $(III, III)$. The Schmid operators are defined by
\begin{equation}\label{schmid}
\left\{
\begin{aligned}
&{\mathscr D}^{J, L}_{\tau}: \quad C^\infty_{\tau}(K\backslash G)\to C^\infty_{\tau^-_\Lambda}(K\backslash G),\quad \phi\mapsto P^J_\tau \circ (\nabla^L\phi),\\
&{\mathscr D}^{J^\vee, R}_{\tau^\vee}: \quad  C^\infty_{\tau^\vee}(G/K)\to C^\infty_{\tau^-_{\Lambda^\vee}}(G/K),\quad \phi\mapsto P^{J^\vee}_{\tau^\vee}\circ(\nabla^R\phi)
\end{aligned}\right.
 \end{equation}
where $\tau^-_{\Lambda^\vee}$ and $P^{J^\vee}_{\tau^\vee}$ are defined in a manner similar to $\tau^-_\Lambda$ and $P^J_\tau$, i.e. defined by (\ref{tau-}) and (\ref{pj}) respectively with $\Lambda, J$ replaced by $\Lambda^\vee, J^\vee$.
Since $\pi_\lambda^\vee$ contains its minimal $K$-type $\tau^\vee$ with multiplicity one, for each $(\frak{g}, K)$-embedding $\iota: \pi_\lambda^\vee\hookrightarrow C^\infty_\tau(K\backslash G)$ there corresponds to a unique $\phi_\iota \in C^\infty_\tau(K\backslash G)$ such that 
\[
\iota(v^\vee)=\langle \phi_\iota(\cdot), v^\vee\rangle
\]
for all $v^\vee\in \tau^\vee$.
In \cite{Y} it is proved that $\iota\mapsto \phi_\iota$ gives an isomorphism
\[
 \textrm{Hom}_{(\frak{g},K)}(\pi^\vee_\lambda, C^\infty_{\tau}(K\backslash G))\cong\textrm{Ker}( {\mathscr D}^{J, L}_{\tau})
\]
and similarly there is another isomorphism
\[
 \textrm{Hom}_{(\frak{g},K)}(\pi_\lambda, C^\infty_{\tau^\vee}(G/K))\cong\textrm{Ker}( {\mathscr D}^{J^\vee, R}_{\tau^\vee}).
\]
This pair of isomorphisms is crucial for our computation. It follows that $\phi_\lambda$ lies in both kernel spaces, hence it can be determined by solving differential equations.

\subsection{Differential equations for spherical functions}

We shall calculate the $A$-radial part of the Schmid operator. Let us choose the following orthogonal basis of $\frak{p}$,
\[
\{X_{ij}+X_{ji}, \sqrt{-1}(X_{ij}-X_{ji})\}_{i=1,2,~j=3}.
\]
Then we see that $\nabla^{L/R}$ can be decomposed into $\nabla^{L/R}_++\nabla^{L/R}_-$ along the decomposition $\frak{p}_{\mathbb{C}}=\frak{p}_++\frak{p}_-$, where up to a constant we have
\[
\nabla^L_+\phi(a)=\sum_{i=1, 2,~j=3} l_{X_{ji}}\phi(a)\otimes X_{ij},\quad \nabla^L_-\phi(a)=\sum_{i=1, 2,~j=3} l_{X_{ij}}\phi(a)\otimes X_{ji}.
\]
The formulas for $\nabla^R_{\pm}$ are similar. To calculate the left differentiation $l_X$, we need the following lemma (cf. \cite[p.307]{HaKO}).

\begin{lemma}\label{iwasawa}
For $X\in\frak{p}_{\mathbb{C}}$, let $X=\sum_{\beta\in \Sigma^+} X_\beta+X_A +X_K$ be the Iwasawa decomposition. Decompose
$X_\beta=X_\beta^- +X_\beta^+$ along $\frak{p}_{\mathbb{C}} +\frak{k}_{\mathbb{C}}$. Then for $a\in A^+$,
\[
X=X_A+\sum_{\beta\in\Sigma^+} \big(\big[ \sinh \beta(\log a)\big]^{-1}\mathrm{Ad }(a) X_{\beta}^+ - \coth \beta(\log a) X_\beta^+\big).
\]
\end{lemma}

\begin{proof}  We have $X_\beta^-=\frac{1}{2}(X_\beta-\theta X_\beta),$
$X_\beta^+= \frac{1}{2}(X_\beta+\theta X_\beta)$. Since $X\in\frak{p}_{\mathbb{C}}$, we see that
\[
X=X_A+ \sum_{\beta\in \Sigma^+}X_\beta^-,\quad X_K=\sum_{\beta\in\Sigma^+}X_\beta^+.
\]
By [K, Lemma 8.24], it follows that
\[
X_\beta=\frac{2}{1-a^{2\beta}}X_\beta^+ - \frac{2a^\beta}{1-a^{2\beta}} \textrm{Ad}(a)X_\beta^+.
\]
The assertion follows from substituting the last equation into $X=X_A +\sum_\beta (X_\beta - X_\beta^+)$.
\end{proof}

Put $\tau_\pm = \tau\otimes \textrm{Ad}_\pm$. Applying Lemma \ref{iwasawa} to $l_X\phi(a)\otimes Y$ with
$\phi\in C^\infty_{\tau,\tau^\vee}(K\backslash G/ K)$, $Y\in \frak{p}_\pm$, we get
\begin{align}
\label{lx}&~l_X\phi(a)\otimes Y\\
=&~l_{X_A}\phi(a)\otimes Y+\sum_\beta \big(\big[\sinh \beta(\log a)\big]^{-1}l_{\mathrm{Ad }(a) X_{\beta}^+}-\coth \beta(\log a)l_{X_\beta^+}\big)\phi(a)\otimes Y\nonumber\\
=&~l_{X_A}\phi(a)\otimes Y-\sum_\beta \big(\big[\sinh \beta(\log a)\big]^{-1}r_{X_{\beta}^+}+\coth \beta(\log a)l_{X_\beta^+}\big)\phi(a)\otimes Y\nonumber\\
=&~ l_{X_A}\phi(a)\otimes Y + \sum_{\beta} \big(\big[\sinh \beta(\log a)\big]^{-1} \tau^\vee(X_\beta^+)+\coth \beta(\log a)\tau_\pm(X_\beta^+)\big)(\phi(a)\otimes Y)\nonumber\\
&-\sum_\beta \coth \beta(\log a)\phi(a)\otimes [X_\beta^+, Y].\nonumber
\end{align}
In the above, we view $\phi(a)\otimes Y$ as a vector in $(V_\Lambda\otimes \frak{p}_\pm)\otimes V_{\Lambda^\vee}$, $\tau^\vee$ acts on the factor $V_{\Lambda^\vee}$, and $\tau_\pm$ acts on the factor $V_\Lambda\otimes \frak{p}_\pm$. A similar formula for $r_X\phi(a)\otimes Y$ can be obtained if we replace $a$ by $a^{-1}$ in Lemma \ref{iwasawa}
and replace $\tau$ by $\tau^\vee$.

\begin{lemma}\label{iwasawa1}
With above notations,
\begin{align*}
&X_{13, 2\alpha}^+=-\frac{1}{2}Z_{13}, \quad X_{13, A}=\frac{1}{2}H,\quad X_{13, \alpha}=0\\
&X_{31, 2\alpha}^+=\frac{1}{2}Z_{13},\quad X_{31, A}=\frac{1}{2}H,\quad X_{13,\alpha}=0\\
&X_{23, \alpha}^+= -X_{21},\quad X_{23, A}=X_{23, 2\alpha}=0\\
&X_{32, \alpha}^+=X_{12},\quad X_{32,A}=X_{32, 2\alpha}=0.
\end{align*}
\end{lemma}

\begin{proof}
Straightforward computations using (\ref{rootspace}).
\end{proof}

Now we are ready to calculate explicitly the Schmid operators. As a first step, they can be decomposed as

\begin{lemma} For the pair $(J, J^\vee)$, one has the following
\begin{align*}
&(I, II):\quad {\mathscr D}^{I, L}_\tau\phi= P^+_\tau\circ (\nabla^L_-\phi)+P^-_\tau \circ (\nabla^L_-\phi),\quad {\mathscr D}^{II, R}_{\tau^\vee}\phi=P^+_{\tau^\vee}\circ (\nabla^R_+\phi)+P^-_{\tau^\vee} \circ (\nabla^R_+\phi),\\
&(II, I):\quad {\mathscr D}^{II, L}_\tau\phi= P^+_\tau\circ( \nabla^L_+\phi)+P^-_\tau (\circ \nabla^L_+\phi),\quad {\mathscr D}^{I, R}_{\tau^\vee}\phi=P^+_{\tau^\vee}\circ ( \nabla^R_-\phi)+P^-_{\tau^\vee} \circ (\nabla^R_-\phi),\\
&(III, III):\quad{\mathscr D}^{III, L}_\tau\phi= P^-_\tau\circ (\nabla^L_+\phi)+P^-_\tau \circ (\nabla^L_-\phi),\quad {\mathscr D}^{III, R}_{\tau^\vee}\phi=P^-_{\tau^\vee}\circ (\nabla^R_+\phi)+P^-_{\tau^\vee} \circ (\nabla^R_-\phi),
\end{align*}
where $P^\pm_\tau$ are given by $($\ref{projector}$)$ and $P^\pm_{\tau^\vee}$ are defined in a similar way.
\end{lemma}

Hence for each case, $\phi_\lambda(a)$ is annihilated by four operators appearing in above decompositions.  However, it turns out that
\[
P^\pm_{\tau}\circ \nabla^L_{\varepsilon}\phi_\lambda(a)=0 \quad \textit{ is equivalent to } \quad P^\pm_{\tau^\vee}\circ \nabla^R_{-\varepsilon}\phi_\lambda(a)=0,
\]
where $\varepsilon=\pm$. By the lemma, this implies that for $U(2,1)$ one has
\[
{\mathscr D}^{J, L}_\tau\phi_\lambda(a)=0\quad \textit{ is equivalent to } \quad {\mathscr D}^{J^\vee, R}_{\tau^\vee}\phi_\lambda(a)=0.
\]
Then we only need to give our results for ${\mathscr D}^{J, L}_\tau$. In the sequel we put $a=a_t=\exp(tH)$ as in Section 3.1, and by abuse of notation write $\phi_\lambda(t)=\phi_\lambda(a_t)$, $c_i(t)=c_i(a_t)$ etc.

\begin{proposition}
The equation $P^+_\tau\circ \nabla^L_-\phi_\lambda(t)=0$ reads
\[
\sum^r_{i=0}\bigg[\bigg(-\frac{1}{2}\frac{d}{dt}-\frac{i+s}{2}\tanh t +i \coth t\bigg)c_i(t)+(r-i+1)(\sinh t)^{-1}c_{i-1}(t)\bigg]f^{r+1,s-2}_i\otimes f^{\Lambda^\vee}_{r-i}=0,
\]
hence implies that
\begin{equation}\label{+-}
\bigg(-\frac{1}{2}\frac{d}{dt}-\frac{i+s}{2}\tanh t +i \coth t\bigg)c_i(t)+(r-i+1)(\sinh t)^{-1}c_{i-1}(t)=0,\quad i=0,\ldots, r.
\end{equation}
\end{proposition}

\begin{proof} Since the projection $P^+_\tau$ is $K$-equivariant,  for $X\in\frak{k}_{\mathbb{C}}$, $Y\in\frak{p}_-$ one has
\[
P_\tau^+\big(\tau_-(X)(f^\Lambda_i \otimes Y)\big)=\tau_{r+1,s-2}(X)P_\tau^+(f^\Lambda_i\otimes Y).
\]
By Lemma \ref{proj}, (\ref{kaction}) (\ref{ad}) (\ref{lx}) and Lemma \ref{iwasawa1}, it follows that
\begin{eqnarray}
\label{x31}&& P^+_\tau \big(l_{X_{13}}\phi_\lambda(t)\otimes X_{31}\big)\\
&=&\frac{1}{2}P^+_\tau \sum_i \big\{\big[l_H-(\sinh 2t)^{-1}\tau^\vee(Z_{13})-\coth 2t \cdot \tau_-(Z_{13})\big]c_i(t) (f^\Lambda_i \otimes f^{1,-2}_0)\otimes f^{\Lambda^\vee}_{r-i}\nonumber\\
&& +\coth 2t \cdot c_i(t)f^\Lambda_i\otimes [Z_{13}, X_{31}] \otimes f^{\Lambda^\vee}_{r-i}\big\}\nonumber\\
&=& -\frac{1}{2}\sum_i \bigg[\frac{d}{dt}+(\sinh 2t)^{-1}\tau^\vee(Z_{13})+\coth 2t \big(\tau_{r+1,s-2}(Z_{13})+2\big)\bigg]c_i(t)f^{r+1,s-2}_i\otimes f^{\Lambda^\vee}_{r-i}
\nonumber\\
&=&-\frac{1}{2}\sum_i\bigg[\frac{d}{dt}-(i+s)(\sinh 2t)^{-1}+(i+s)\coth 2t\bigg]c_i(t)f^{r+1,s-2}_i\otimes f^{\Lambda^\vee}_{r-i}
\nonumber\\
&=&-\frac{1}{2}\bigg[\frac{d}{dt}+(i+s)\tanh t \bigg]c_i(t)f^{r+1,s-2}_i\otimes f^{\Lambda^\vee}_{r-i},\nonumber
\end{eqnarray}
noting that $Z_{13}$ acts as a scalar,  $X_{31}=f^{1,-2}_0$ and $[Z_{13}, X_{31}]=-2X_{31}=-2f^{1,-2}_0$.

Similarly, we have
\begin{align}
\label{x32}&~ P^+_\tau \big(l_{X_{23}}\phi_\lambda(t)\otimes X_{32}\big)\\
=&~ P^+_\tau\sum_i\big\{ \big[(\sinh t)^{-1}\tau^\vee(X_{21})+\coth t \cdot \tau_-(X_{21})\big]c_i(t)(f_i^\Lambda\otimes f^{1,-2}_1)\otimes f^{\Lambda^\vee}_{r-i} \nonumber\\
&~+\coth t \cdot c_i(t)f_i^\Lambda\otimes[X_{21}, X_{32}] \otimes f^{\Lambda^\vee}_{r-i}\big\}\nonumber\\
=&~ \sum_i \big\{\big[(\sinh t)^{-1}\tau^\vee(X_{21})+\coth t \cdot \tau_{r+1, s-2}(X_{21})]c_i(t) f^{r+1, s-2}_{i+1}\otimes f^{\Lambda^\vee}_{r-i}\nonumber\\
&~-\coth t \cdot c_i(t) f^{r+1, s-2}_i\otimes
f^{\Lambda^\vee}_{r-i}\big\}\nonumber\\
=&~ \sum_i \big[ i\coth t \cdot c_i(t) f_i^{r+1,s-2}\otimes f^{\Lambda^\vee}_{r-i} +(r-i)(\sinh t)^{-1} c_i(t) f^{r+1,s-2}_{i+1}\otimes f^{\Lambda^\vee}_{r-i-1}\big],\nonumber
\end{align}
noting that $X_{32}=-f^{1,-2}_1$ and $[X_{21}, X_{32}]=-X_{31}=-f^{1,-2}_0$. Adding (\ref{x31}) and (\ref{x32}), then (\ref{+-}) follows from the fact that the coefficient of
$f^{r+1,s-2}_i\otimes f^{\Lambda^\vee}_{r-i}$ is zero.
\end{proof}

Analogous calculations show that

\begin{proposition}
$P^-_\tau\circ \nabla^L_-\phi_\lambda(t)=0$ reads
\begin{equation}\label{--}
i\bigg(\frac{1}{2}\frac{d}{dt}+\frac{i+s}{2}\tanh t +(r+1-i)\coth t\bigg)c_i(t)+(r+1-i)^2(\sinh t)^{-1}c_{i-1}(t)=0;
\end{equation}
$P^+_\tau\circ \nabla^L_+\phi_\lambda(t)=0$ reads
\begin{equation}\label{++}
\bigg(-\frac{1}{2}\frac{d}{dt}+\frac{i+s}{2}\tanh t +(r-i) \coth t\bigg)c_i(t)+(i+1)(\sinh t)^{-1}c_{i+1}(t)=0;
\end{equation}
$P^-_\tau\circ \nabla^L_+\phi_\lambda(t)=0$ reads
\begin{equation}\label{-+}
(r-i)\bigg(\frac{1}{2}\frac{d}{dt}-\frac{i+s}{2}\tanh t +(i+1)\coth t\bigg)c_i(t)+(i+1)^2(\sinh t)^{-1}c_{i+1}(t)=0
\end{equation}
with $i=0,\ldots, r$.
\end{proposition}

\subsection{Solutions using hypergeometric functions}

Let us give the explicit formulas for $c_i(t)$ and $\psi_\lambda(t)$ by solving differential equations.
For convenience we put
\[
\tilde{c}_i(t)=\langle f^\Lambda_i, f^{\Lambda^\vee}_{r-i}\rangle c_i(t)=(-1)^i{r\choose i}^{-1}c_i(t).
\]
Then $
\psi_\lambda(t)=\sum\limits^r_{i=0}\tilde{c}_i(t)$. We first deal with the (anti-)holomorphic discrete series.

\begin{theorem}\label{hol}
$\mathrm{(i)}$ For $\lambda\in\Xi_{I}$,
\[\tilde{c}_i(t)=(\cosh t)^{-i-s},\quad i=0,\ldots, r,\quad\psi_\lambda(t)=\sum^r_{i=0}(\cosh t)^{-i-s}.\]
$\mathrm{(ii)}$ For $\lambda\in \Xi_{II}$,
\[
\tilde{c}_i(t)=(\cosh t)^{i+s},\quad i=0,\ldots, r,\quad
\psi_\lambda(t)=\sum^r_{i=0}(\cosh t)^{i+s}.\]
\end{theorem}

\begin{proof}
If $\lambda\in\Xi_{I}$, then we have (\ref{+-}) and (\ref{--}), which imply that
\begin{equation}
\label{dc1}
\left\{
\begin{aligned}
&\frac{d\tilde{c}_i}{dt}+(i+s)\tanh t \cdot \tilde{c}_i(t)=0,\\
& \tilde{c}_{i-1}(t)=\cosh t \cdot \tilde{c}_i(t).
\end{aligned}\right.
\end{equation}
By (\ref{iso}) we have the initial condition $\tilde{c}_i(0)=1$. Thus by separation of variables we get the desired solution for (\ref{dc1}).
If $\lambda\in\Xi_{II}$ then we have (\ref{++}) and (\ref{-+}), and the computation is similar.
\end{proof}

The rest of this subsection is devoted to finding $\psi_\lambda$ for middle discrete series, i.e. the case $\lambda\in\Xi_{III}$. Then we have
to solve (\ref{--}) and (\ref{-+}). As preparations let us recall some basic facts about Riemann's differential equation and
hypergeometric functions (cf. \cite{WW}). Riemann's $P$-symbol
\begin{equation}\label{riemann}
u=P\begin{bmatrix} a & b & c &\\ \alpha & \beta & \gamma & z\\  \alpha' & \beta' & \gamma' & \end{bmatrix}
\end{equation}
denotes that $u$ is a solution of the second-order differential equation in $z$
\begin{align*}
&\frac{d^2u}{dz^2}+\bigg[\frac{1-\alpha-\alpha'}{z-a}+\frac{1-\beta-\beta'}{z-b} + \frac{1-\gamma-\gamma'}{z-c}\bigg]\frac{du}{dz}\\
& + \bigg[\frac{\alpha\alpha'(a-b)(a-c)}{z-a}+\frac{\beta\beta'(b-a)(b-c)}{z-b}+\frac{\gamma\gamma'(c-a)(c-b)}{z-c}\bigg]\frac{u}{(z-a)(z-b)(z-c)}=0,\nonumber
\end{align*}
where $a, b, c$ are three distinct singularities, and $\alpha+\alpha'+\beta+\beta'+\gamma+\gamma'=1$. Besides the obvious symmetry with respect to parameters, Riemann's equation satisfies the transformation
\begin{equation}\label{tran}
P\begin{bmatrix} a & b & c &\\ \alpha & \beta & \gamma & z\\  \alpha' & \beta' & \gamma' & \end{bmatrix}=P\begin{bmatrix} a_1 & b_1 & c_1 &\\ \alpha & \beta & \gamma & z_1\\  \alpha' & \beta' & \gamma' & \end{bmatrix},
\end{equation}
where $z_1, a_1, b_1, c_1$ are derived from $z, a, b, c$ by the same fractional linear transformation, i.e. a transformation of the form
$x\mapsto \dfrac{Ax+B}{Cx+D}$.

\

An important class of special functions is the hypergeometric series
\begin{align*}
F(a,b,c,z)=&~\sum^\infty_{n=0}\frac{a(a+1)\cdots (a+n-1) b(b+1)\cdots (b+n-1)}{n!c(c+1)\cdots (c+n-1)}z^n,\\
=&~ \frac{\Gamma(c)}{\Gamma(a)\Gamma(b)}\sum^\infty_{n=0}\frac{\Gamma(a+n)\Gamma(b+n)}{n!\Gamma(c+n)}z^n,
\end{align*}
which converges absolutely for $|z|<1$, provided that $c$ is not zero or a negative integer. For $|z|\geq1$ it can be analytically continued along any path that avoids the branching points $0$ and $1$. Gua{\ss}'s theorem asserts that for $Re(c-a-b)>0$,
\[
F(a,b,c,1)=\frac{\Gamma(c)\Gamma(c-a-b)}{\Gamma(c-a)\Gamma(c-b)}.
\]
If $c-a-b$ is a negative integer, then $F(a,b,c,z)$ has a pole of order $a+b-c$ at $z=1$, which can be seen from the relation
\begin{equation}\label{1-z}
F(a,b,c,z)=(1-z)^{c-a-b}F(c-a,c-b,c,z).
 \end{equation}
 Hypergeometric series includes many familiar mathematical functions as special cases, e.g. $F(1,b,1,z)=(1-z)^{-b}$.

The derivative of hypergeometric series is given by
\begin{equation}\label{deri}
\frac{d}{dz}F(a,b,c,z)=\frac{ab}{c}F(a+1,b+1,c+1,z).
\end{equation}
The following lemma, which we collect for later use, is a typical example of Gau{\ss}'s contiguous relations.

\begin{lemma} \label{hplem}
If $b\neq 0$, then
\[
zF(a+1, b+1, c+1, z)=\frac{c}{b}\big(F(a+1,b,c,z)-F(a,b,c,z)\big).
\]
\end{lemma}

\begin{proof}
By series expansion we have
\begin{align*}
F(a+1, b, c, z)=&~\frac{\Gamma(c)}{\Gamma(a+1)\Gamma(b)}\sum^\infty_{n=0}\frac{\Gamma(a+n+1)\Gamma(b+n)}{n!\Gamma(c+n)}z^n\\
=&~\frac{\Gamma(c)}{\Gamma(a)\Gamma(b)}\sum^\infty_{n=0}\bigg(1+\frac{n}{a}\bigg)\frac{\Gamma(a+n)\Gamma(b+n)}{n!\Gamma(c+n)}z^n\\
=&~F(a,b,c,z)+\frac{z}{a}\frac{d}{dz}F(a,b,c,z).
\end{align*}
The lemma follows from (\ref{deri}) immediately.
\end{proof}




$F(a,b,c,z)$ is a solution of the hypergeometric differential equation
\[
P\begin{bmatrix} 0 & \infty & 1 &\\ 0 & a & 0 & z\\  1-c & b & c-a-b & \end{bmatrix},
\]
which has three singularities $0$, $1$ and $\infty$.
Using transformation properties of Riemann's equation, it can be shown that a solution of (\ref{riemann}) is
\begin{equation}\label{soln}
\bigg(\frac{z-a}{z-b}\bigg)^\alpha\bigg(\frac{z-c}{z-b}\bigg)^\gamma F\bigg(\alpha+\beta+\gamma, \alpha+\beta'+\gamma, 1+\gamma-\gamma', \frac{(a-b)(z-c)}{(a-c)(z-b)}\bigg)
\end{equation}
provided that $\gamma-\gamma'$ is not a negative integer.

By switching parameters one can obtain Kummer's 24 solutions
of (\ref{riemann}) in terms of hypergeometric series, which can be further reduced to 6 linearly independent ones.
For example the solution (\ref{soln}) is given as $u_{17}$ in \cite[Chap XIV 14.3]{WW}.

Getting back to the matrix coefficients $\tilde{c}_i(t)$ and $\psi_\lambda(t)$, we now prove the following

\begin{theorem}\label{mid}
For $\lambda\in\Xi_{III}$,
\begin{align*}
\tilde{c}_i(t)=&~ (\cosh t)^{-2-i+s}F(1+i, 1-s, r+2, \tanh^2 t) \\
=&~(\cosh t)^{-2r-2+i-s}F(1+r-i, 1+r+s, r+2, \tanh^2 t),\quad i=0,\ldots, r,\\
 \psi_\lambda(t)=&~\sum^r_{i=0}\tilde{c}_i(t).
\end{align*}
\end{theorem}


\begin{proof} (Sketch) Introduce a variable $z=(\cosh t)^{-2}$. Combining (\ref{--}) and (\ref{-+}), after a bunch of messy computations one can derive that $\tilde{c}_i$ satisfies the Riemann's equation
 \begin{equation}\label{P1}
 P\begin{bmatrix} 0 & \infty &  1 & \\ 1+\dfrac{i-s}{2} &  \dfrac{i+s}{2} & 0 &  z \\ r+1-\dfrac{i-s}{2} & -\dfrac{i+s}{2} & -r-1 &
 \end{bmatrix}.
 \end{equation}
Let $z_1=\dfrac{1-z}{1+z}$, then by transformation rule (\ref{tran}) we see that $\tilde{c}_i$ satisfies
\[
P\begin{bmatrix} 1 & -1 & 0 & \\ 1+\dfrac{i-s}{2} &  \dfrac{i+s}{2} & 0 &  z_1 \\ r+1-\dfrac{i-s}{2} & -\dfrac{i+s}{2} & -r-1 &
 \end{bmatrix}.
\]
Thus from (\ref{soln}) we get a solution of (\ref{P1})
\begin{align}
\label{soln1} &~\bigg(\frac{z_1-1}{z_1+1}\bigg)^{1+(i-s)/2}F\bigg(1+i, 1-s, r+2, \frac{2z_1}{z_1+1}\bigg)\\
=&~(-z)^{1+(i-s)/2}F\bigg(1+i, 1-s, r+2, 1-z\bigg)\nonumber\\
=&~(-1)^{1+(i-s)/2} (\cosh t)^{-2-i+s}F(1+i, 1-s, r+2, \tanh^2 t).\nonumber
\end{align}
Since $\tanh^2 t<1$, the solution has a convergent series expansion. We claim that, up to constant it is the unique solution
of (\ref{P1}) that is regular and nonvanishing at $z=1$, and vanishing at $z=0$, corresponding to $t=0$ and $t=\infty$ respectively.
To see the vanishing at $z=0$, note that for $\lambda\in\Xi_{III}$ one has $\lambda=\Lambda$, $r+s=\lambda_1-\lambda_3>0$. Thus if $F(1+i, 1-s, r+2, 1-z)$ has a pole at $z=0$, the order is at most
\[
i-s-r < \frac{i-s}{2} < 1+\frac{i-s}{2}.
\]
This implies the claim. Since the matrix coefficient $\tilde{c}_i(t)$ has the same analytic properties as stated in above claim, it must be a multiple
of (\ref{soln1}) satisfying the condition
$\tilde{c}_i(0)=1$. 

This proves the first formula for $\tilde{c}_i(t)$. The second one follows from (\ref{1-z}), or from switching the parameters $\{\alpha,\beta\}$ with $\{\alpha',\beta'\}$ in (\ref{P1}).
\end{proof}

The two equivalent formulas for $\tilde{c}_i(t)$ in Theorem \ref{mid} can be used for different cases in Section 5.

\subsection{General formulas}

We close this section by giving the full formulas for the matrix coefficients of discrete series $\pi_\lambda$ of $U(2,1)$.
Most of previous notations and results, with slight modifications when necessary, still work in this case. In particular,
the representation $\tau_\Lambda$ of $S(U(2)\times U(1))$ extends to $U(2)\times U(1)$ and we have (see Section 3.2)
\[
\tau_\Lambda\cong \big(\rho_r\otimes\det{}^{\Lambda_2}\big)\boxtimes\chi^{\Lambda_3}.
\]
We still use the basis $f^\Lambda_i$, $i=0,\ldots, r$ of $V_\Lambda$ and the matrix coefficient $\psi_\lambda$, whose $A$-radial part is
\[
\psi_\lambda(t)=\sum^r_{i=0}c_i(t)\langle f^\Lambda_i, f^{\Lambda^\vee}_{r-i}\rangle.
\]
We also write $g=k'^{-1}a_tk$ with $k$, $k'$ being parameterized as in Section 3.3. Then by spherical properties of $\phi_\lambda$ we get
\[
\psi_\lambda(g)=\psi_\lambda(kk'^{-1}a_t)=\sum^r_{i=0}c_i(t)\langle\tau(kk'^{-1})f^\Lambda_i, f^{\Lambda^\vee}_{r-i}\rangle=\sum^r_{i=0}\tilde{c}_i(t)\langle\tau(kk'^{-1})f^\Lambda_i, h^\Lambda_i\rangle.
\]
The function $\tilde{c}_i(t)$ is given in Section 4.4, and we have the following straightforward computation for the pairing $\langle\tau(kk'^{-1})f^\Lambda_i, h^\Lambda_i\rangle$.

We have
\begin{align*}
kk'^{-1}=&~\textrm{diag}(1,\zeta,1)\kappa(\theta)\textrm{diag}(\xi\bar{\xi}',\eta\bar{\eta}',\gamma\bar{\gamma}')\kappa(-\theta')\textrm{diag}(1,\bar{\zeta}',1)\\
=&~\begin{bmatrix}
a & \bar{\zeta}'b & 0\\
\zeta c & \zeta\bar{\zeta}'d & 0\\
0 & 0 & \gamma\bar{\gamma}'
\end{bmatrix},
\end{align*}
where
\begin{align*}
&a=\xi\bar{\xi}'\cos\theta\cos\theta'+\eta\bar{\eta}'\sin\theta\sin\theta',\\
&b=-\xi\bar{\xi}'\cos\theta\sin\theta'+\eta\bar{\eta}'\sin\theta\cos\theta',\\
&c=-\xi\bar{\xi}'\sin\theta\cos\theta'+\eta\bar{\eta}'\cos\theta\sin\theta',\\
&d=\xi\bar{\xi}'\sin\theta\sin\theta'+\eta\bar{\eta}'\cos\theta\cos\theta'.
\end{align*}
Recall that by our convention $f^\Lambda_i=x^i y^{r-i}$, with dual basis $h^\Lambda_i$. Thus
\begin{align*}
&\langle\tau(kk'^{-1})f_i^\Lambda, h^\Lambda_i\rangle\\
=&~\big(\zeta\bar{\zeta}'\xi\bar{\xi}'\eta\bar{\eta}'\big)^{\Lambda_2}\big(\gamma\bar{\gamma}'\big)^{\Lambda_3}\langle (ax+\zeta cy)^i(\bar{\zeta}'bx+\zeta\bar{\zeta}'dy)^{r-i}, h^\Lambda_i\rangle \\
=&~\big(\zeta\bar{\zeta}'\big)^{\Lambda_1-i}\big(\xi\bar{\xi}'\eta\bar{\eta}'\big)^{\Lambda_2}\big(\gamma\bar{\gamma}'\big)^{\Lambda_3}\sum^i_{l=0}{i\choose l}{r-i \choose i-l}a^l(bc)^{i-l}d^{r-2i+l}.
\end{align*}
In summary, we may write 
\begin{equation}\label{pairing}
\psi_\lambda(g)=\sum^r_{i=0}\tilde{c}_i(t)\big(\zeta\bar{\zeta}'\big)^{\Lambda_1-i}\big(\xi\bar{\xi}'\eta\bar{\eta}'\big)^{\Lambda_2}\big(\gamma\bar{\gamma}'\big)^{\Lambda_3}\sum^i_{l=0}{i\choose l}{r-i \choose i-l}a^l(bc)^{i-l}d^{r-2i+l}.
\end{equation}
We notice the following useful identity
\begin{equation}\label{id}
\sum^i_{l=0}{i\choose l}{r-i \choose i-l}={r\choose i},
\end{equation}
which is a special case of 
\[
\sum_{l=0}^{j}{m \choose l} {n\choose j-l} ={m+n \choose j}.
\]

\section{Explicit values of zeta integrals}

In this section we shall apply the results from previous sections to calculate the zeta integrals on $G=U(2,1)$.
As in the introduction, let $\pi=\pi_\lambda$ be a discrete series of $\widetilde{G}$ with minimal $\widetilde{K}$-type $\sigma$ and Harish-Chandra parameter $\lambda$, $\pi'=\theta(\pi^\vee)$ be the theta lifting of $\pi^\vee$
with minimal $\widetilde{K}'$-type $\sigma'$, where $\pi^\vee$ is the contragredient of $\pi$ with minimal $\widetilde{K}$-type $\sigma^\vee$. In Section 2 we write $(\sigma, \sigma')$ for the pair of weights under theta correspondence, but now we change to  $(\sigma^\vee, \sigma')$ and we hope this change of notation does not cause much confusion. Recall that we have chosen a nonzero vector $\phi\in\mathscr{H}_{\sigma^\vee,\sigma'}\subset \mathscr{H}$ in the space of joint harmonics.  Let $\psi_{\pi}=\psi_\lambda$ be the canonical matrix coefficient of $\pi$ defined in the fashion of (\ref{canonical}). Then our purpose is to calculate the explicit value of the zeta integral
\begin{equation}\label{zeta}
\int_G(\omega(g)\phi,\phi)\cdot \psi_{\pi}(g)dg.
\end{equation}
Note that both $(\omega(g)\phi,\phi)$ and $\psi_\pi(g)$ are functions on $\widetilde{G}$, but their product descends to $G$, so above integration (\ref{zeta}) makes sense.

In this section, by abuse of notation  we write $\sigma$ for 
both the Blatter parameter $\Lambda$ and the minimal $K$-type $\tau_\Lambda$ in Section 4. For convenience we also put $\sigma=(\Lambda_1, \Lambda_2, \Lambda_3)$, then we have the two parameters $r, s$ defined by (\ref{rs}). The space of Harish-Chandra parameters for $\widetilde{G}$ is
\[
\widetilde{\Xi}_c=\{\lambda=(\lambda_1,\lambda_2,\lambda_3)\in \sqrt{-1}\frak{t}^*: \lambda_j\equiv\frac{1}{2}\textrm{ mod }\mathbb{Z}, ~\lambda\textrm{ is }\Delta\textrm{-regular and }\Delta_c^+\textrm{-dominant}\}.
\]
Put $\widetilde{\Xi}_J=\{\lambda\in\widetilde{\Xi}_c: \lambda\textrm{ is }\Delta_J^+\textrm{-dominant}\}$. Then our formulas for the matrix coefficients of discrete series in Section 4 still hold, up to a factor of $\pm1$.

Let $\hat{G}$ denote the Langlands dual group of $G$ and fix a Gal$(\mathbb{C}/\mathbb{R})$-splitting $(\mathcal{B},\mathcal{T}, \{\mathcal{X}_{\check{\alpha}}\}_{\check{\alpha}\in\check{\Delta}})$ of $\hat{G}$, where $\check{\Delta}$ is the set of roots of $\mathcal{T}$ in $\hat{G}$.
Then $\mathcal{B}$ determines a set of positive roots $\check{\Delta}^+$.
By Corollary of  \cite[Lemma 23.1]{H-C}, up to a constant which depends on the choice of the Haar measure on $G$, the discrete series $\pi_\lambda$ has formal degree
\begin{equation}\label{fdegree}
d(\pi_\lambda)=\prod_{\check{\alpha}\in\check{\Delta}^+}\big|\langle\lambda,\check{\alpha}\rangle\big|.
\end{equation}
It is obviously independent of the choice of the splitting. For the measure on $G=U(2,1)$ chosen in Section 3.1, we have exactly (\ref{fdegree}). More precisely, $\hat{G}=GL_3(\mathbb{C})$, and if
$\lambda=(\lambda_1,\lambda_2,\lambda_3)$ then
\[
d(\pi_\lambda)=|(\lambda_1-\lambda_2)(\lambda_2-\lambda_3)(\lambda_1-\lambda_3)|.
\]

We have the following results from \cite{HLS}. If $G$ or $G'$ is compact, then by Schur's lemma for any $\phi\in\mathscr{H}_{\sigma^\vee,\sigma'}$,
\begin{equation}\label{cpt}
\int_G (\omega(g)\phi,\phi) \cdot \psi_{\pi}(g)dg =\frac{\|\phi\|^2}{d(\pi_\lambda)}.
\end{equation}
In general, let $P_\pi: \omega\to \sigma^\vee\otimes \pi'$ be the orthogonal projection, then
\begin{equation}\label{ncpt}
\int_G(\omega(g)\phi,\phi) \cdot \psi_{\pi}(g)dg =\frac{\|P_\pi\phi\|^2}{d(\pi_\lambda)}.
\end{equation}
In \cite{L1} it is proved that $P_\pi$ is nonzero on $\mathscr{H}_{\sigma^\vee, \sigma'}$. Essentially, our results give the ratio $\|P_\pi\phi\|/\|\phi\|$, which is a positive constant independent of the choice of $\phi$, due to Schur's lemma. Note that this is the constant denoted by $c_{\psi, V,\lambda}$ in the introduction.

Now we  state and prove the main results of this section. Because of (\ref{cpt}), we are primarily interested in the situation where $G'$ is non-compact, i.e. computing (\ref{ncpt}) explicitly for cases (C) and (D). But we remark that our treatment is unform and it is not hard to verify (\ref{cpt}) for cases (A) and (B) using the formulas in Subsection 3.3 and Section 4. Due to the discrepancy between (anti-)holomorphic and middle discrete series, and also between the corresponding matrix coefficients of
the Fock model, we sketch the calculations for (C1) and (C2) case by case. Those for case (D) are very similar and will be omitted.

\begin{thmA}
Assume that $\sigma=(-\mu_2-\frac{3}{2}, -\mu_1-\frac{3}{2}, \nu+\frac{3}{2})$. Then
$\lambda=(-\mu_2-\frac{1}{2}, -\mu_1-\frac{3}{2}, \nu+\frac{1}{2})\in \widetilde{\Xi}_{ II}$, 
\[
\int_G(\omega(g)\phi,\phi)\cdot \psi_{\pi}(g)dg=\frac{\|\phi\|^2}{(\mu_1-\mu_2+1)(\mu_1+\nu+2)(\mu_2+\nu+1)}.
\]
 \end{thmA}

\begin{thmB}
Assume that $\sigma=(\nu_1+\frac{3}{2},\nu_2+\frac{3}{2},-\alpha-\frac{3}{2})$. Then
$\lambda=(\nu_1+\frac{3}{2}, \nu_2+\frac{1}{2}, -\alpha-\frac{1}{2})\in \widetilde{\Xi}_{ I}$, 
\[
\int_G(\omega(g)\phi,\phi) \cdot \psi_{\pi}(g)dg=\frac{\|\phi\|^2}{(\nu_1-\nu_2+1)(\nu_1+\alpha+2)(\nu_2+\alpha+1)}.
\]
\end{thmB}

\begin{thmC1} Assume that $\sigma=(-\mu_2-\frac{1}{2},-\mu_1-\frac{1}{2}, -\alpha+\frac{1}{2})$. Then we have three subcases:
\\
(C1, I) If $\alpha\geq \mu_1+4$, then $\lambda=(-\mu_2-\frac{1}{2},-\mu_1-\frac{3}{2}, -\alpha+\frac{3}{2})\in\widetilde{\Xi}_{ I}$,
\[
\int_G(\omega(g)\phi,\phi) \cdot \psi_{\pi}(g)dg=\frac{\|\phi\|^2}{(\mu_1-\mu_2+1)(\alpha-1)\alpha}.
\]
(C1, II) If $\mu_2\geq \alpha+2$, then
$\lambda=\left(-\mu_2+\frac{1}{2},-\mu_1-\frac{1}{2}, -\alpha-\frac{1}{2}\right)\in\widetilde{\Xi}_{ II}$,
\[
\int_G(\omega(g)\phi,\phi) \cdot \psi_{\pi}(g)dg=\frac{\|\phi\|^2}{(\mu_1-\mu_2+1)(\mu_2+\alpha)(\mu_1+\alpha+1)}.
\]
(C1, III) If $\mu_1\geq\alpha\geq \mu_2+2$, then
 $\lambda=\sigma\in\widetilde{\Xi}_{ III}$,
\[
\int_G(\omega(g)\phi,\phi) \cdot \psi_{\pi}(g)dg=\frac{\|\phi\|^2}{(\mu_1-\mu_2+1)(\mu_1+1)\alpha}.
\]
\end{thmC1}


\begin{proof} We have $r=\mu_1-\mu_2$, $s=\alpha-\mu_1-1$. Combining previous results, we see that the zeta integral equals
\[
\int^\infty_0\int_{K \times  K} ( \omega(a_t k)\phi, \omega(k')\phi)\sum^r_{i=0}\tilde{c}_i(t)\langle\sigma(kk'^{-1})f_i^{\sigma}, h^\sigma_i\rangle dk  dk' D(t)dt,
\]
where $( \omega(a_t k)\phi, \omega(k')\phi)$, $\langle\sigma(kk'^{-1})f_i^{\sigma}, h^\sigma_i\rangle$ and $D(t)$ are given by (\ref{fockc1}), (\ref{pairing}) and (\ref{da}) respectively. We only need to consider the terms in the integrand that have nonzero contributions to the integral. Due to the simple fact $\int_{U(1)} u^n du=\delta_{n,0}$, we only need to keep the terms where $\eta\bar{\eta}'$ has zero exponent. By comparing the exponent of $\eta\bar{\eta}'$ in (\ref{fockc1}) and (\ref{pairing}), we see that the zeta integral is reduced to
\begin{align*}
&\sum^r_{i=0}\int^\infty_0(\cosh t)^{-\mu_2-\alpha-3}\tilde{c}_i(t)D(t) dt \int_{[0,\pi/2]^2}\int_{U(1)^2} (\cosh^{-1}t \cos \theta \cos \theta' + \zeta\bar{\zeta}'\sin \theta \sin \theta')^r\\
&\hspace{1cm} \times (\zeta\bar{\zeta}')^{-i} {r\choose i} (\sin\theta\sin\theta')^i (\cos\theta\cos\theta')^{r-i}\sin2\theta\sin2\theta' d\zeta d\zeta' d\theta d\theta'  \|\phi\|^2.
\end{align*}
Similarly, by considering the exponent of $\zeta\bar{\zeta}'$, above integral reads
\[
\sum^r_{i=0}\int^\infty_0(\cosh t)^{i-\mu_1-\alpha-3}\tilde{c}_i(t)D(t)dt \cdot {r\choose i}^2\bigg(\int_0^{\pi/2}\sin^{2i}\theta \cos^{2(r-i)}\theta\sin 2\theta d\theta\bigg)^2\|\phi\|^2.
\]
Using beta functions we see that
\[
\int_0^{\pi/2}\sin^{2i}\theta \cos^{2(r-i)}\theta\sin 2\theta d\theta=\frac{i! (r-i)!}{(r+1)!}.
\]
Hence the zeta integral is equal to
\begin{equation}\label{c11}
\frac{1}{(r+1)^2}\sum^r_{i=0}\int^\infty_{0}(\cosh t)^{i-\mu_1-\alpha-3}\tilde{c}_i(t)D(t)dt \|\phi\|^2.
\end{equation}
Note that $D(t)dt=\sinh^2 t\sinh 2t dt= (\cosh^2 t-1) d\cosh^2 t$, for $n>2$ one has
\[
\int^\infty_0 (\cosh t)^{-2n}D(t)dt= \frac{1}{(n-2)(n-1)}.
\]
We now consider the three types of discrete series.

(C1, I) With $\tilde{c}_i(t)=(\cosh t)^{-i-s}$ given by Theorem \ref{hol} (i), (\ref{c11})  equals
\[
\frac{1}{(r+1)^2}\sum^r_{i=0}\int^\infty_0(\cosh t)^{-2\alpha-2}D(t)dt \|\phi\|^2=\frac{1}{(r+1)(\alpha-1)\alpha}\|\phi\|^2.
\]

(C1, II) With $\tilde{c}_i(t)=(\cosh t)^{i+s}$ given by Theorem \ref{hol} (ii), (\ref{c11}) equals
\begin{align*}
&~\frac{1}{(r+1)^2}\sum^r_{i=0}\int^\infty_0 (\cosh t)^{2i-2\mu_1-2\alpha-4}D(t)dt\|\phi\|^2 \\
=&~\frac{1}{(r+1)^2}\sum^r_{i=0} \frac{1}{(\mu_1-i+\alpha)(\mu_1-i+\alpha+1)}\|\phi\|^2\\
=&~\frac{1}{(r+1)^2}\bigg(\frac{1}{\mu_2+\alpha}-\frac{1}{\mu_1+\alpha+1}\bigg)\|\phi\|^2\\
=&~\frac{1}{(r+1)(\mu_2+\alpha)(\mu_1+\alpha+1)}\|\phi\|^2.
\end{align*}

(C1, III) Using the first formula for $\tilde{c}_i(t)$ in Theorem \ref{mid}, (\ref{c11}) equals
\begin{align*}
&~\frac{1}{(r+1)^2}\sum^r_{i=0}\int^\infty_0(\cosh t)^{-2\mu_1-6}F(1+i, 1-s, r+2, \tanh^2 t)D(t)dt\|\phi\|^2\\
=&~\frac{1}{(r+1)^2}\sum^r_{i=0}\int^1_0 (1-z)^{\mu_1}zF(1+i,1-s, r+2,z) dz  \|\phi\|^2\\
=&-\frac{1}{s(r+1)}\sum^r_{i=0}\int^1_0 (1-z)^{\mu_1} \big(F(1+i,-s,r+1,z)-F(i,-s,r+1,z)\big)dz\|\phi\|^2\\
=&-\frac{1}{s(r+1)}\int^1_0(1-z)^\mu_1\big(F(r+1,-s,r+1,z)-1\big)dz\|\phi\|^2\\
=&-\frac{1}{s(r+1)}\int^1_0\big((1-z)^{\mu_1+s}-(1-z)^{\mu_1}\big)dz\|\phi\|^2\\
=&~\frac{1}{(r+1)(\mu_1+1)(\mu_1+s+1)}\|\phi\|^2,
\end{align*}
where we make change of variable $z=\tanh^2 t$ in the first equality, and apply Lemma \ref{hplem} in the second equality.

The theorem follows from substituting $r=\mu_1-\mu_2$, $s=\alpha-\mu_1-1$.
\end{proof}

\begin{thmC2} Assume that  $\sigma=(\nu-\frac{1}{2},-\mu-\frac{1}{2}, \beta+\frac{1}{2})$. Then we have two subcases:
\\
(C2, II) If $\beta\geq \nu+2$, then $\lambda=(\nu+\frac{1}{2}, -\mu-\frac{1}{2}, \beta-\frac{1}{2})\in\widetilde{\Xi}_{ II}$, 
\[
\int_G(\omega(g)\phi,\phi) \cdot \psi_{\pi}(g)dg=\frac{\|\phi\|^2}{(\nu+\mu+1)(\beta+\mu+1)\beta}.
\]
(C2, III) If $\nu\geq \beta+2$, then $\lambda=\sigma\in\widetilde{\Xi}_{ III}$, 
\[
\int_G(\omega(g)\phi,\phi) \cdot \psi_{\pi}(g)dg=\frac{\|\phi\|^2}{\nu(\nu+\mu+1)(\beta+\mu+1)}.
\]
\end{thmC2}

\begin{proof} We have $r=\mu+\nu$, $s=-\mu-\beta-1$. In this case we use (\ref{fockc2}). Similar to the proof of Theorem C1, considering the exponents of $\eta\bar{\eta}'$ yields the following
iterated integral
\begin{align}\label{c22}
&\sum^r_{i=0}{r\choose i}\int^\infty_0 (\cosh t)^{-\beta-3}\tilde{c}_i(t)D(t)dt\\
&\qquad\times\int_{[0,\pi/2]^2} J_i(\theta,\theta',t) (\sin\theta\sin\theta')^i (\cos\theta\cos\theta')^{r-i}\sin2\theta\sin2\theta'd\theta d\theta' \|\phi\|^2,\nonumber
\end{align}
where
\begin{align*}
J_i(\theta,\theta',t)=&\int_{U(1)^2} (\cosh^{-1} t \cos \theta \cos \theta' + \zeta\bar{\zeta}'\sin \theta \sin \theta')^\mu\\
&\hspace{1cm}\times (\cosh^{-1}t \sin\theta \sin \theta' + \bar{\zeta}\zeta'\cos\theta \cos\theta' )^\nu(\zeta\bar{\zeta}')^{\nu-i}d\zeta d\zeta'\\
=&\sum^\mu_{j=0}(\cosh t)^{2j-\mu-i}{\mu\choose j}{\nu\choose i-j}(\sin\theta\sin\theta')^i(\cos\theta\cos\theta')^{r-i}.
\end{align*}
After integrating $\theta$ and $\theta'$,  (\ref{c22}) reads
\begin{equation}\label{c22'}
\frac{r!}{[(r+1)!]^2}\sum^r_{i=0}\sum^\mu_{j=0}  {\mu\choose j}{\nu\choose i-j}i!(r-i)!\int^\infty_0(\cosh t)^{2j-\mu-i-\beta-3}\tilde{c}_i(t)D(t)dt \|\phi\|^2.
\end{equation}

(C2, II) With $\tilde{c}_i(t)=(\cosh t)^{i+s}$ given by Theorem \ref{hol} (ii), (\ref{c22'}) equals
\begin{align*}
&~\frac{r!}{[(r+1)!]^2}\sum^r_{i=0}\sum^\mu_{j=0}  {\mu\choose j}{\nu\choose i-j}i!(r-i)!\int^\infty_0(\cosh t)^{2j-2\mu-2\beta-4}D(t)dt \|\phi\|^2\\
=&~\frac{r!}{[(r+1)!]^2}\sum^\mu_{j=0}{\mu\choose j}\sum^r_{i=0}{\nu\choose i-j}i!(r-i)!\frac{1}{(\mu-j+\beta)(\mu-j+\beta+1)}\|\phi\|^2.
\end{align*}
Making a change of variable $k=i-j$, by Lemma \ref{comb} for $j=0,\ldots,\mu$ we have
\[
{\mu\choose j}\sum^r_{i=0}{\nu\choose i-j}i!(r-i)!={\mu\choose j}\sum^\nu_{k=0}{\nu\choose k}(k+j)!(\nu-k+\mu-j)!=\frac{(r+1)!}{\mu+1}.
\]
Hence we end up with the value
\begin{align*}
&~\frac{1}{(r+1)(\mu+1)}\sum^\mu_{j=0}\frac{1}{(\mu-j+\beta)(\mu-j+\beta+1)}\|\phi\|^2\\
=&~\frac{1}{(r+1)(\mu+1)}\bigg(\frac{1}{\beta}-\frac{1}{\beta+\mu+1}\bigg)\|\phi\|^2\\
=&~\frac{1}{(r+1)(\beta+\mu+1)\beta}\|\phi\|^2.
\end{align*}

(C2, III) Using the second formula for $\tilde{c}_i(t)$ in Theorem \ref{mid}, (\ref{c22'}) equals
\begin{align*}
&\frac{r!}{[(r+1)!]^2}\sum^r_{i=0}\sum^\mu_{j=0}  {\mu\choose j}{\nu\choose i-j}i!(r-i)!\\
&\quad\times \int^\infty_0(\cosh t)^{2j-2r-4}F(1+r-i,1+r+s,r+2, \tanh^2 t)D(t)dt \|\phi\|^2.
\end{align*}
Making changes of variables $k=i-j$, $z=\tanh^2 t$, above equals $I(\mu,\nu,\beta)\|\phi\|^2$, where
\begin{align*}
I(\mu,\nu,\beta)=&\frac{r!}{[(r+1)!]^2}\sum^\mu_{j=0}\sum^\nu_{k=0}{\mu\choose j}{\nu\choose k} (j+k)!(r-j-k)!\\
&\times\int^1_0 (1-z)^{r-1-j} z F(1+r-j-k, 1+r+s, r+2,z)dz.
\end{align*}
To evaluate $I(\mu,\nu,\beta)$, we recall  Euler's integral representation for hypergeometric functions
\[
\mathrm{B}(a,c-a)F(a,b,c,z)=\int^1_0 x^{a-1}(1-x)^{c-a-1}(1-zx)^{-b}dx,\quad Re(c)>Re(a)>0,
\]
provided $|z|<1$ or $|z|=1$ and both sides converge. Here B$(\cdot,\cdot)$ is the beta function. Applying this formula we obtain
\begin{align*}
I(\mu,\nu,\beta)&=\frac{1}{r+1}\sum^\mu_{j=0}\sum^\nu_{k=0}{\mu \choose j}{\nu \choose k}\int^1_0\int^1_0(1-z)^{r-1-j}z x^{r-j-k}(1-x)^{j+k}(1-zx)^{-1-r-s}dxdz\\
&=\frac{1}{r+1}\int^1_0\int^1_0(1-z)^{r-1}zx^r\left(1+\frac{1-x}{x(1-z)}\right)^\mu \left(1+\frac{1-x}{x}\right)^\nu (1-zx)^{-1-r-s}dxdz\\
&=\frac{1}{r+1}\int^1_0\int^1_0 (1-z)^{\nu-1}z(1-zx)^{\beta+\mu-\nu}dxdz.
\end{align*}
This double integral can be evaluated easily, although there are two different cases according to  whether $\beta+\mu-\nu$ equals $-1$ or not. In both cases one finds that
\[
I(\mu,\nu,\beta)=\frac{1}{\nu(\nu+\mu+1)(\beta+\mu+1)}.
\]
This finishes the calculation in the case (C2, III).
\end{proof}

\begin{thmD1} Assume that $\sigma=(\nu_1+\frac{1}{2},\nu_2+\frac{1}{2}, \beta-\frac{1}{2})$. Then we have three subcases:
\\
(D1, I) If $\nu_2\geq\beta+2$, the $\lambda=(\nu_1+\frac{1}{2},\nu_2-\frac{1}{2},\beta+\frac{1}{2})\in\widetilde{\Xi}_{I}$,
\[
\int_G(\omega(g)\phi,\phi)\cdot \psi_{\pi}(g)dg=\frac{\|\phi\|^2}{(\nu_1-\nu_2+1)\nu_2(\nu_1+1)}.
\]
(D1, II) If $\beta\geq \nu_1+4$, then $\lambda=(\nu_1+\frac{3}{2},\nu_2+\frac{1}{2},\beta-\frac{3}{2})\in\widetilde{\Xi}_{ II}$,
\[
\int_G(\omega(g)\phi,\phi)\cdot \psi_{\pi}(g)dg=\frac{\|\phi\|^2}{(\nu_1-\nu_2+1)(\beta-1)\beta}.
\]
(D1, III) If $\nu_1\geq\beta\geq\nu_2+2$, then $\lambda=\sigma\in\widetilde{\Xi}_{ III}$,
\[
\int_G(\omega(g)\phi,\phi) \cdot \psi_{\pi}(g)dg=\frac{\|\phi\|^2}{(\nu_1-\nu_2+1)(\nu_1+1)\beta}.
\]
\end{thmD1}

\begin{thmD2} Assume that $\sigma=(\nu+\frac{1}{2},-\mu+\frac{1}{2}, -\alpha-\frac{1}{2})$. Then we have two subcases:
\\
(D2, I) If $\alpha\geq\mu+2$, then $\lambda=(\nu+\frac{1}{2},-\mu-\frac{1}{2}, -\alpha+\frac{1}{2})\in \widetilde{\Xi}_{ I}$,
\[
\int_G(\omega(g)\phi,\phi) \cdot \psi_{\pi}(g)dg=\frac{\|\phi\|^2}{(\mu+\nu+1)(\alpha+\nu+1)\alpha}.
\]
\\
(D2, III) If $\mu\geq\alpha+2$, then $\lambda=\sigma\in\widetilde{\Xi}_{ III}$,
\[
\int_G(\omega(g)\phi,\phi) \cdot \psi_{\pi}(g)dg=\frac{\|\phi\|^2}{\mu(\mu+\nu+1)(\alpha+\nu+1)}.
\]
\end{thmD2}

\end{document}